\documentclass[10pt,twocolumn,twoside]{IEEEtran}

\usepackage[top=1in, bottom=1in, left=1in, right=1in]{geometry}

\usepackage{subfigure,float,amsmath,amsfonts,amssymb,color,graphicx}

\newtheorem{theorem}{\indent {Theorem}}

\newtheorem{lemma}{\indent {Lemma}}
\newtheorem{remark}{\indent {Remark}}

\newtheorem{corollary}{\indent {Corollary}}
\def\R{\mathbb{R}}
\def\1{{\bf 1}}
\def\a{\alpha}
\def\e{{\epsilon}}

\def\aor#1{{{#1}}}

\newenvironment{proof}[1][Proof]{\noindent\textbf{#1.} }{\ \rule{0.5em}{0.5em}}
\def\ao#1{{{#1}}}

\def\e{\epsilon}
\def\bx{{\mathbf x}}
\def\bw{{\mathbf w}}
\def\bz{{\mathbf z}}

\def\bg{{\mathbf g}}
\def\bv{{\mathbf v}}

\def\Argmin{\mathop{\rm Argmin}}

\begin{document}

\title{Distributed optimization over time-varying directed graphs}
\author{Angelia Nedi\'c and Alex Olshevsky\thanks{
The authors are with the Department of Industrial and 
Enterprise Systems Engineering, University of Illinois at Urbana-Champaign, 104 S. Mathews Avenue, Urbana IL, 61801, Emails: $\{$angelia,aolshev2$\}$@illinois.edu. A. Nedi\'c gratefully acknowledges the support by the NSF grants CMMI 07-42538 and CCF 11-11342, and by the ONR Navy Basic Research Challenge N00014-12-1-0998.}
}
\maketitle

\begin{abstract} We consider distributed optimization by a collection of nodes, each having access to its own convex function, whose collective goal is to minimize the sum of the functions. 
The communications between nodes are described
by a time-varying sequence of {\em directed} graphs, which is uniformly strongly connected. For such communications, assuming that every node knows its out-degree, 
we develop a broadcast-based algorithm, termed 
the {\em subgradient-push}, which steers every node to an optimal value under a standard assumption of subgradient boundedness.  The subgradient-push requires
no knowledge of either the number of agents or the graph sequence to implement. 
Our analysis shows that the subgradient-push algorithm converges at a rate 
of $O \left( \ln t/\sqrt{t} \right)$. The proportionality constant in the convergence rate 
depends on the initial values at the nodes, the subgradient norms and, more interestingly, 
on both the speed of the network information diffusion 
and the imbalances of influence among the nodes.  \end{abstract} 

\section{Introduction} 
We consider the problem of distributed convex optimization by a network of nodes when knowledge of the
objective function is scattered throughout the network and unavailable at any singe location. There has been much recent interest in multi-agent optimization problems of this type 
that arise whenever a large collections of nodes - which may be processors, nodes of a sensor network, vehicles, or UAVs - desire to collectively
optimize a global objective by means of local actions taken by each node without any centralized coordination. 

Specifically, we will study the problem of optimizing a sum of $n$ convex functions by a network of $n$ nodes when each function is known to only 
 a single node. This problem frequently arises when control and signal processing protocols need to be implemented in sensor networks. For example, 
the problems including robust statistical inference \cite{rabbat}, formation control \cite{othesis}, non-autonomous power control \cite{ram_info}, 
distributed message
routing \cite{neglia}, and spectrum access coordination \cite{li-han}, 
can be reduced to variations of this problem. 
We will be focusing on the case when communication between nodes is {\em directed} and {\em time-varying}. 

Distributed optimization of a sum of convex functions has received a surge  of interest in recent years \cite{AN2009, rabbat, johansson, LopesSayed2007, Lobel2011, LobelOF2011, SN2011, nedic-broadcast,ChenSayed2012, Ram2012, Duchi2012}. There is now a considerable theory justifying the
use of distributed subgradient methods in this setting, and their performance limitations and convergence times are well-understood. Moreover, distributed subgradient methods have been used to 
propose new solutions for a number of problems in distributed control and sensor networks \cite{ram_info,neglia,li-han}.
However, the works cited above assumed communications among nodes are either fixed or undirected.

Our paper is the first to demonstrate a working subgradient protocol in the
setting of directed time-varying communications. We develop a broadcast-based protocol, termed the {\em subgradient-push}, which steers every node to an optimal value under
a standard assumption of subgradient boundedness.  The subgradient-push requires each node to know its
out-degree at all times, but beyond this it needs
no knowledge of the graph sequence of even of the number of agents to implement. Our results show that it converges at a rate 
of $O \left( \ln t/\sqrt{t} \right)$, where the constant depends, among other factors, on the information diffusion speed of the corresponding directed graph sequence and 
a measure of the influence imbalance among the nodes.

Our work is closest to the recent papers \cite{rabbat_allerton2012, rabbat_cdc2012,Tsianos2011,Tsianos2013, Gh-Cortes}. \aor{The papers \cite{rabbat_allerton2012, rabbat_cdc2012,Tsianos2011,Tsianos2013} proposed a distributed subgradient algorithm which is very
similar to the one we study here, involving the introduction of subgradients into an information aggregation procedure known as ``push-sum.'' The 
convergence of this scheme for directed but fixed topologies was shown in  ~\cite{rabbat_allerton2012, rabbat_cdc2012,Tsianos2011,Tsianos2013};
 implementation of the protocol proposed in these papers} appears to require knowledge of the
graph or of the number of agents. By contrast, our results work in time-varying networks and are fully distributed,  requiring no knowledge of either the graph sequence or the number of agents.  The 
paper \cite{Gh-Cortes} shows the convergence of a distributed optimization protocol in continuous time, also for directed but fixed graphs; moreover, an additional assumption is
made in \cite{Gh-Cortes} that the graph is ``balanced.''  

All the prior work in distributed optimization, except for \cite{rabbat_allerton2012, rabbat_cdc2012,Tsianos2011,Tsianos2013}, requires time-varying communications with some form of balancedness, often reflected in a requirement of having a sequence of doubly stochastic matrices that are
commensurate with the sequence of underlying  communication graphs. In contrast, our proposed method
removes the need for the doubly stochastic matrices. 
The proposed distributed optimization model is motivated by applications that are characterized by time-varying directed communications,
such as those arising in a mobile sensor network where the links among 
nodes will come and go as nodes move in and out of line-of-sight or broadcast range of 
each other. Moreover, if different nodes are capable of broadcasting messages at different power levels, then communication links connecting the nodes
will necessarily be unidirectional. 


The remainder of this paper is organized as follows.
We begin in Section~\ref{sec:results} 
where we describe the problem of interest, outline the subgradient-push algorithm, and \ao{state} the main convergence results. \ao{Section~\ref{sec:averaging} is devoted to the proof of a key lemma, namely the convergence rate result for a 
perturbed version of the so-called push-sum protocol; this lemma is then used in the subsequent} proofs of convergence and \ao{convergence rate} 
for the subgradient-push in Section \ref{sec:optim}. 
 Finally, some conclusions are offered in Section \ref{sec:concl}.

\noindent
{\bf Notation}: 
We use boldface to distinguish between the vectors in $\R^d$ and scalars \aor{or vectors of different dimensions.} 
For example, the vector $\bx_i(t)$ is in boldface while the scalar $y_i(t)$ is not.
The vectors such as $y(t)\in\R^n$ obtained by stacking scalar values $y_i(t)$ associated with the nodes 
\aor{are thus} not bolded. 
Additionally, for a vector $y$, we will also occasionally use $[y]_j$ to denote its $j$'th entry.
For a matrix $A$, we will use $A_{ij}$ or $[A]_{ij}$ to denote its $i,j$'th entry. \aor{The notation $A'$ will refer to the transpose of the matrix
$A$. }
The vectors are seen as column vectors unless otherwise explicitly stated.
We use $\1$ to denote the vector of ones,  
and $\| y\|$ for the Euclidean norm of a vector~$y$. 

\section{Problem, Algorithm and Main Results\label{sec:results}} 
We consider a network of $n$ nodes whose goal is to solve distributedly the following 
minimization problem: 
\[ \hbox{minimize } F(\bz) \triangleq \sum_{i=1}^n f_i( \bz) \quad\hbox{over $\bz\in\mathbb{R}^d$},\] 
where only node $i$ knows the convex function $f_i: \R^d \to \R$. Under the 
assumption that the set of optimal solutions $Z^* = \Argmin_{\bz \in \R^d} F(\bz)$ is nonempty, 
we would like to design a protocol by which all agents maintain variables 
$\bz_i(t)$ converging to the same point in $Z^*$ with time. 

We will assume that, at each time $t$, {\it node $i$ can only send messages to its out-neighbors in some
directed graph $G(t)$}. 
Naturally, the graph $G(t)$ will have vertex set $\{1, \ldots, n\}$, and we will use $E(t)$ to 
denote its edge set. 
Also, naturally, the sequence $\{G(t)\}$ should posses some good long-term connectivity properties. 
A standard assumption, which
we will be making, is that the sequence $\{G(t)\}$ is uniformly strongly connected (or, as it is sometimes called, $B$-strongly-connected), namely,
that there exists some ineger $B>0$ (possibly unknown to the nodes) such that the graph with edge set 
\[E_B(k) =  \bigcup_{i=kB}^{(k+1)B-1} E(i) \] 
is strongly connected for every $k \geq 0$. This is a typical assumption for many results in multi-agent control: it is considerably
weaker than requiring each $G(t)$ be connected for it allows the edges necessary 
for connectivity to appear over a long time period and in
arbitrary order; however, it is still strong enough to derive bounds on the 
speed of information propagation from one part of the network to another.

Finally, we introduce the notation $N^{\rm in}_i(t)$ and $N^{\rm out}_i(t)$ for the in- and out-neighborhoods 
of node $i$, respectively,  at time $t$. We will \aor{require} these neighborhoods to include the node $i$ itself\footnote{Alternatively, one may define these neighborhoods in a standard way of the graph theory, but require that  each graph in the sequence $\{G(t)\}$ has a self-loop at every node.}; formally, we have
\begin{eqnarray*} N^{\rm in}_i(t) & = & \{j\mid (j,i)\in E(t)\} \cup\{i\}, \\ 
N^{\rm out}_i(t) & = & \{j\mid (i,j)\in E(t)\} \cup\{i\}, 
\end{eqnarray*} 
and $d_i(t)$ for the out-degree of node $i$, i.e., 
\[d_i(t) = |N^{\rm out}_i(t)|.\] 
Crucially,  we will be assuming that {\it every node $i$ 
knows its out-degree $d_i(t)$ at every time $t$}. 

Our main contribution is the design of an algorithm which successfully accomplishes the task of distributed 
minimization of $F(\bz)$  assumptions we have laid out above.  
Our scheme is a combination of the subgradient method and the so-called 
{\em push-sum} protocol, recently studied in the papers \cite{benezit, dominguez, dobra-kempe, ciblat}. 
We will refer to our protocol as the {\em subgradient-push} method. 
A subgradient method using the push-sum protocol has been considered 
in~\cite{Tsianos2011,rabbat_allerton2012,rabbat_cdc2012,Tsianos2013} 
for a fixed communication network, whose implementation 
requires some knowledge of the graph or the number of agents. In contrast, our algorithm can handle time-varying networks and its implementation requires node $i$ knowing its local out-degree $d_i(t)$ only.

\subsection{The subgradient-push method} 
Every node $i$ maintains
vector variables $\bx_i(t), \bw_i(t)\in\R^d$,  as well as a scalar variable 
$y_i(t)$. These
quantities are updated according to the following rules:
for all $t\ge0$ and all $i=1,\ldots,n$,
\begin{eqnarray}\label{eq:minmet} 
\bw_i(t+1) & = & \sum_{j \in N_i^{\rm in}(t)} \frac{\bx_j(t)}{d_j(t)},\cr
&&\hbox{}\cr
y_i(t+1) & = & \sum_{j \in N_i^{\rm in}(t)} \frac{y_j(t)}{d_j(t)}, \cr
&&\hbox{}\cr
\bz_{i}(t+1) & = & \frac{\bw_{i}(t+1)}{y_{i}(t+1)},\cr
&&\hbox{}\cr
\bx_i(t+1) &=& \bw_i(t+1) - \alpha(t+1) \bg_i(t+1),  
\end{eqnarray} where $\bg_i(t+1)$ is a subgradient of the function
$f_i(\bz)$ at $\bz=\bz_i(t+1)$. The method is initiated with an arbitrary vector
$\bx_i(0)\in\R^d$ at node $i$, and with $y_i(0)=1$ for all $i$. 
The stepsize $\alpha(t+1)>0$ satisfies the following decay conditions 
\begin{align}\label{eq:dimstep}
 \sum_{t=1}^{\infty} \alpha(t) =\infty,
\qquad \sum_{t=1}^{\infty} \alpha^2(t) < \infty, \cr 
\alpha(t) \le \alpha(s) \ \mbox{for all } t > s\ge1.
\end{align} 
We note that the above equations have simple broadcast-based implementation: each node $j$ broadcasts the 
quantities $\bx_j(t)/d_j(t), y_j(t)/d_j(t)$ to all of the nodes $i$ in its out-neighborhood\footnote{We note that we make use here of the 
assumption that node $i$ knows its out-degree $d_i(t)$.}, which simply sum all the messages 
they receive to obtain
$\bw_i(t+1)$ and $y_i(t+1)$. 
The update equations for $\bz_i(t+1), \bx_i(t+1)$ can be executed without any further communications 
\aor{among} nodes during step~$t$. 

Without the subgradient term in the final equation, our protocol would be a version of the push-sum protocol \cite{dobra-kempe} for average computation 
 studied recently in \cite{benezit, dominguez, ciblat}. For intuition on the precise form of these equations, we refer the reader to these three papers; roughly
 speaking, the somewhat involved form of the updates is intended to ensure that every node receives an equal weighting after all the linear combinations and
 ratios have been taken. In this case the vectors $\bz_{i}(t+1)$ converge to some common point, i.e.,
 a consensus is achieved.  
The inclusion of the subgradient terms in the updates of $\bx_i(t+1)$ is intended to steer the consensus point 
 towards the optimal set $Z^*$, while
the push-sum updates steer the vectors $\bz_{i}(t+1)$ towards each other. 
Our main results, which we describe in the next section, demonstrate that this scheme succeeds  in steering 
 all vectors $\bz_{i}(t+1)$ towards the same point in the solution set~$Z^*$.
 
 \subsection{Our results}
Our first theorem demonstrates the correctness of the subgradient-push method for an arbitrary stepsize $\alpha(t)$ satisfying Eq.~(\ref{eq:dimstep}); this holds under the assumptions we have laid out above, as well as an additional technical assumption on the subgradient boundedness. 
\begin{theorem} \label{mainthm} Suppose that:
\begin{enumerate} \item[(a)] 
The graph sequence $\{G(t)\}$ is uniformly strongly connected. 
\item[(b)] 
Each function $f_i(\bz)$ is convex \aor{over $\mathbb{R}^d$} 
and the set $Z^* = \Argmin_{\bz \in \R^d} F(\bz)$ is nonempty. 
\item[(c)] The subgradients of each $f_i(\bz)$ are uniformly bounded, i.e., there exists $L_i < \infty$ such that
for all $\bz\in\mathbb{R}^d$,
\[
\aor{\|\bg_i\|}\le L_i\ \hbox{for all subgradients $\bg_i$ of $f_i(\bz)$}.
\] 
\end{enumerate} Then, the distributed subgradient-push method of Eq.~\eqref{eq:minmet} with the stepsize satisfying the conditions in Eq.~\eqref{eq:dimstep} has the following property
\[ \lim_{t \rightarrow \infty} \bz_i(t) = \bz^* \qquad\mbox{ for all $i$ and for some $\bz^* \in Z^*$}. \]
\end{theorem} 

Our second theorem makes explicit the rate at which the objective function converges to its optimal value. 
As standard with subgradient
methods, we will make two tweaks in order to get a convergence rate result: 
(i) we take a stepsize which decays as $\alpha(t) = 1/\sqrt{t}$ (stepsizes which decay at faster rates
usually produce inferior convergence rates), 
and (ii) each node $i$ will maintain a convex combination of the values $\bz_i(1), \bz_i(2), \ldots$ for which the convergence rate will be obtained. 
We then demonstrate that the subgradient-push converges at a rate of $O(\ln t/\sqrt{t})$; this is formally stated in the following theorem. The theorem makes use of the 
matrix $A(t)$ that captures the weights used in the construction of $\bw_i(t+1)$ and $y_i(t+1)$ in 
Eq.~\eqref{eq:minmet}, which are defined by
\begin{equation}\label{eq:defA}
A_{ij}(t) =\left\{
\begin{array}{ll}
1/d_j(t) & \hbox{whenever $j \in N_i^{\rm in}(t)$},\cr 
0 &\hbox{otherwise}.\end{array}\right.
\end{equation}  \aor{Moreover, we will use the notation $$L = \sum_{i=1}^n L_i.$$}

\begin{theorem} \label{convrate} 
Suppose that all the assumptions of Theorem \ref{mainthm} hold, and let 
$\alpha(t) = 1/\sqrt{t}$ for $t\ge1$.  Moreover, suppose that every node
$i$ maintains the variable $\widetilde \bz_i(t) \in \R^d$ initialized at time $t=0$ with any
$\widetilde \bz_i(0)\in\R^d$ and updated by 
\[ \widetilde \bz_i(t+1) 
= \frac{\alpha(t+1) \bz_i(t+1) + S(t) \widetilde \bz_i(t)}{S(t+1)}
\quad\hbox{for $t\ge0$},  \] 
where $S(0)=0$ and $S(t) = \sum_{s=0}^{t -1} \alpha(s+1)$ for $t\ge1$.  
Then, we have for all $t \geq 1$, $i=1, \ldots, n$, and any $\bz^* \in Z^*$,
\begin{eqnarray*}  
&& F \left( \widetilde \bz_i(t+1) \right) - F(\bz^*) 
\leq  \frac{n}{2}\frac{\|\aor{\bar\bx(0)} - \bz^*\|_1}{\sqrt{t+1}} \cr
&&+ \frac{\aor{L^2} \left( 1 + \ln (t+1) \right)}{2 n\sqrt{t+1} } \cr
&&\hbox{}\cr
 & & + \frac{ 24 \aor{L} \sum_{j=1}^n \|\bx_j(0)\|_1 }
 {   \delta (1-\lambda) \sqrt{t+1} }\cr 
 &&\hbox{}\cr
 &&+ \frac{24 \aor{d} \aor{L^2} \left( 1+ \ln t \right) }{\delta (1-\lambda) \sqrt{t+1} }, 
 \end{eqnarray*}
 where 
 \[ \bar \bx(0) = \frac{1}{n} \sum_{i=1}^n \bx_i(0) .\] 
The scalars $\lambda$ and $\delta$ are functions of the graph sequence $G(1), G(2),\ldots,$ 
which are given by
\begin{align*} \delta \geq  \frac{1}{n^{nB}}, \qquad
\lambda \leq \left( 1 - \frac{1}{n^{nB}} \right)^{1/(nB)}.
\end{align*}
If each of the graphs $G(t)$, $t\ge1$, is regular\footnote{A directed graph $G(t)$ is regular if  
every out-degree and every in-degree of a node in $G(t)$ equals $d(t)$ for some $d(t)$.}, then 
\begin{align*} 
\delta  =  1, \  
\lambda \leq 
\min\left\{\left( 1 - \frac{1}{4n^3} \right)^{1/B} , \max_{t\ge0} \aor{\sigma_2(A(t))} \right\},
\end{align*} where $A(t)$ is defined by Eq.~\eqref{eq:defA}
and $\sigma_2(A)$ is the second-largest singular value of a matrix $A$.
\end{theorem}

Theorem~\ref{convrate} implies that, along the time-averages $\widetilde \bz_i(t)$ for each node $i$,
the network objective function $F(z)$ converges to the optimal objective value $F^*$,
i.e.,
\[\lim_{t\to\infty}F \left( \widetilde \bz_i(t) \right) = F^*\qquad\hbox{for all }i.\]
However, the theorem
does not state anything about the convergence of the sequences $\{\widetilde \bz_i(t)\}$ for $i=1,\ldots,n$.

Theorem~\ref{convrate} provides the rate at which $F \left( \widetilde \bz_i(t)\right)$ converges to 
$F^*$ for any $i$, which is expected. Specifically, 
it is standard for a distributed subgradient method 
to converge at the rate of $O( \ln t/\sqrt{t})$ with the 
constant depending on the subgradient-norm upper bounds $L_i$ and the initial conditions 
$\bx_i(0)$ \cite{Ram2010,Duchi2012}. Moreover, it is also standard for the 
convergence rate of the distributed methods over networks 
to depend on some measure of the connectivity of 
the directed graph sequence $G(1), G(2), \ldots$. Namely, here, the closeness of 
$\lambda$ to $1$ measures the speed at which the (connectivity) graph sequence 
$\{G(t)\}$ diffuses the information among the nodes over time. 
\aor{Additionally}, our rate results also include the parameter $\delta$, 
which is a measure of the imbalance of  influences among the nodes, as we will later see. 
Time-varying directed regular networks 
are uniform in influence and will have $\delta=1$, so that $\delta$ will disappear from the bounds entirely; however, networks which have a row very 
close to zero, \aor{corresponding nodes which are only weakly influenced by all others,  }
will suffer a corresponding blow-up in the convergence time of the subgradient-push algorithm. \aor{Consequently, $\delta$ may be thought of as
a measure of the uniformity of long-term influence among the nodes.}

Moreover, while the term $1/(\delta(1-\lambda))$ appearing in our rate estimate is bounded 
exponentially by $n^{2nB}$ in the worst case, the term need not be this large
for every graph sequence. Indeed,  Theorem~\ref{convrate} 
shows that for a class of time-varying regular directed graphs, $1/(\delta(1-\lambda))$ scales polynomially in $n$.  Our work therefore motivates the search for effective bounds on consensus speed
and the influence imbalances for time-varying directed graphs.  
Finally, we note that previous research \cite{AN2009, Ram2010, Duchi2012} has studied the case when
the matrices $A(t)$ in~\eqref{eq:defA} 
are doubly stochastic, which occurs when the directed graph sequence $\{G(t)\}$ is regular. 
In this case, our polynomial bounds match the previously known results.

\aor{We conclude by briefly summarizing the idea of the proofs as well as the organization of the 
remainder of this paper. As previously remarked, our protocol is a perturbation of the so-called ``push-sum'' protocol for averaging 
studied in \cite{benezit, dominguez, dobra-kempe}. 
We begin in Section \ref{sec:averaging} by showing that as a 
consequence of well-known facts about consensus protocols, such perturbed push-sum protocols are guaranteed to converge when the perturbations are well behaved (in a sense). 
In the subsequent Section \ref{sec:optim}, we interpret the subgradient-push algorithm as a special
perturbation of the push-sum protocol and use the convergence results of Section \ref{sec:averaging}
to show that, after a transient 
period, our subgradient-push algorithm begins to approximate a centralized subgradient scheme. 
Finally some simulations are performed in
Section \ref{sec:simul} and conclusions are drawn in Section \ref{sec:concl}.  }

\section{Perturbed Push-Sum Protocol}\label{sec:averaging}
This section is dedicated to the analysis of a perturbed version of the so-called push-sum protocol, originally introduced in the groundbreaking work~\cite{dobra-kempe} and recently
analyzed in time-varying directed graphs in~\cite{benezit, dominguez}. 
The push-sum is a protocol for node interaction which allows nodes
to compute averages and other aggregates in the network with directed communication links.

The work in~\cite{dobra-kempe, benezit, dominguez} 
demonstrates the convergence of the push-sum protocol. 
Here, we generalize this result 
by showing that the protocol remains convergent even if the state of the nodes is perturbed at each step, as long as the perturbations decay to zero. However, while the iterate sequences (at the nodes) obtained 
by the push-sum method converge to the average of the initial values of the nodes, the iterate sequences produced by the perturbed push-sum method need not converge to a common point; instead, they converge to each other over time.
We will later use this result to prove Theorems \ref{mainthm} and \ref{convrate}. 
Since the result has a self-contained interpretation and 
analysis, we sequester it to this section. 

We begin with a statement of the perturbed push-sum update rule. Every node $i$ maintains scalar 
variables $x_i(t), y_i(t), z_i(t), w_i(t)$, where $y_i(0)=1$ for all $i$. 
These variables are updated as  follows: for $t\ge0$,
\begin{eqnarray}\label{eq:newx}
 w_i(t+1)  &=& \sum_{j \in N^{\rm in}_i(t) } \frac{x_j(t)}{d_j(t)},\cr
 \hbox{}\cr
y_i(t+1) & = & \sum_{j \in N^{\rm in} _i(t)} \frac{y_j(t)}{d_j(t)}, \cr
\hbox{}\cr
z_i(t+1) & = & \frac{w_i(t+1)}{y_i(t+1)},  \cr
\hbox{}\cr
x_i(t+1)&=&w_i(t+1) + \e_i(t+1),
\end{eqnarray} 
where $\e_i(t)$ is a perturbation at time $t$, 
perhaps adversarially chosen. \aor{Recall} that $N_i^{\rm in}(t)$ is the in-neighborhood 
of node $i$ in a directed graph $G(t)$ and $d_j(t)$ is the out-degree of node $j$, as  defined
in Section~\ref{sec:results}.

Without the perturbation term $\e_i(t)$, the method in Eq.~(\ref{eq:newx}) reduces to the push-sum protocol. Moreover, our subgradient-push method of Eq.~(\ref{eq:minmet}) is simply a vector-space analog of
the perturbed pus-sum method of Eq.~(\ref{eq:newx}) with a specific choice of the perturbations. 

The intuition behind the push-sum equations of Eq.~(\ref{eq:newx}) is somewhat involved. This dynamic has been  introduced for the purpose of average computation (in the 
case when all the perturbations $\e_i(t)$ are zero) and has a simple motivating intuition. The push-sum is a variation of a consensus-like protocol wherein every node updates 
its values by taking linear combinations of the values of its neighbors. Due to taking linear combinations, 
some nodes are bound to be more influential than others (meaning that other nodes end 
up placing larger weights on their information), for example by virtue of being
more centrally placed. To cancel out the effect of these influence imbalances of the nodes,
the ratios $z_i(t)=w_i(t)/y_i(t)$  are used, which ensures
that each $z_i(t)$ converges to $(1/n) \sum_{i=1}^n x_i(0).$ We refer the
reader to \cite{dobra-kempe, dominguez, benezit} for more details. 

Next, we rewrite the perturbed push-sum equations more compactly by 
using the definition of the matrix $A(t)$ in Eq.~(\ref{eq:defA}).  
Then, the relations in Eq.~(\ref{eq:newx}) assume the following form:
for all $t\ge0$,
\begin{eqnarray}\label{eq:pert}
 w(t+1) &= &A(t)x(t),\cr
\hbox{}\cr
y(t+1) &=& A(t)y(t), \cr
\hbox{}\cr
z_i(t+1) &=& \frac{w_i(t+1)}{y_i(t+1)}\quad\hbox{for all }i=1,\ldots,n,  \cr
\hbox{}\cr
x(t+1) &=& w(t+1)+\e(t+1),
\end{eqnarray} 
where $\e(t)=(\e_1(t),\ldots,\e_n(t))'$. 
\aor{Each of matrices} $A(t)$ is column-stochastic but not necessarily row-stochastic. 

We are concerned with demonstrating a convergence result and a convergence rate for the 
updates given in Eq.~(\ref{eq:newx}), or equivalently Eq.~\eqref{eq:pert}. 
Specifically,  the bulk of this section is dedicated to 
proving the following lemma.

\begin{lemma}\label{lemma:newx} 
Consider the sequences $\{z_i(t)\}$, $i=1,\ldots,n,$ generated by the method in 
Eq.~\eqref{eq:newx}. Assuming that the graph sequence $\{G(t)\}$ is uniformly
strongly connected, 
the following statements hold: 
\begin{itemize}
\item[(a)] 
For all $t\ge 1$ we have
\begin{eqnarray*}
\left|z_i(t+1) -\frac{\1' x(t)}{n}\right|
& \le &
\frac{8}{\delta}\,\left(\lambda^t\|x(0)\|_1 \right.
\cr
&& \left.
+ \sum_{s=1}^{t}\lambda^{t-s}\|\e(s)\|_1 \right),
\end{eqnarray*} 
where $\delta>0$ and $\lambda \in (0,1)$ satisfy
\[ \delta \geq \frac{1}{n^{nB}}, ~~~ \lambda \leq \left( 1-\frac{1}{n^{nB}} \right)^{1/B}. \] 
If each of the matrices $A(t)$ in Eq.~\eqref{eq:defA} are doubly stochastic, then
\[ \delta = 1, \ 
\lambda \leq \left\{\left( 1 - \frac{1}{\ao{4} n^3} \right)^{1/B}, \max_{t\ge0} \aor{\sigma_2(A(t))}\right\}.\] 
\item[(b)]
If $\lim_{t\to\infty}\e_i(t)=0$ for all $i=1,\ldots,n$, then 
\[\lim_{t\to\infty}\left|z_i(t+1) -\frac{\1' x(t)}{n}\right|=0
\quad\hbox{for all }i.\]
\item[(c$)$]
If $\{\a(t)\}$ is a non-increasing positive scalar sequence with
$\sum_{t=1}^\infty \a(t)|\e_i(t)| <\infty$ for all $i$, then 
\[ \sum_{t=0}^\infty\a(t+1) \left|z_i(t+1) - \frac{\1'x(t)}{n}\right|<\infty
 \ \hbox{for all } i.\]
\end{itemize}
\end{lemma}
For part (b) of Lemma~\ref{lemma:newx}, observe that each of the matrices $A(t)$ is doubly stochastic if 
each of the graphs $G(t)$ is regular.
Furthermore, observe that if $\epsilon_i(t)=0$, this lemma implies that the push-sum method 
converges at a geometric rate. In this case, it is easy to see that $\1' x(t)/n = \1'x(0)/n$ 
and, therefore  $z_i(t) \rightarrow \1' x(0)/n$, so that the push-sum protocol successfully computes 
the average of the initial values. When the perturbations are nonzero, Lemma~\ref{lemma:newx}
states that if these perturbations decay to zero, then the push-sum 
method still converges. Of course, it will no longer be true 
that the convergence is necessarily to the average of the initial values. 

We will prove a series of auxiliary lemmas before beginning the proof of Lemma \ref{lemma:newx}.  
We first remark that the matrices $A(t)$ have a special structure that allows us to efficiently analyze their products. Specifically, we have the following properties of the matrices $A(t)$ 
(see \cite{bhot, Morse, Moreau, Nedic-Ozdaglar-Delays, TBA} for proofs of this and similar statements). 
 
\begin{lemma}\label{lemma:crate} 
Suppose that the graph sequence $\{G(t)\}$ is uniformly strongly connected. 
\aor{Then for each integer $s \geq 0$, there is a stochastic vector $\phi(s)$ such that for all $i,j$ and $t \geq s,$ 
\[|[A'(t)A'(t-1)\cdots A'(s+1)A'(s)]_{ij}-\phi_j(s)|\le C\lambda^{t-s},\]}
for some $C$ and $\lambda \in (0,1)$. 
\end{lemma} 

\smallskip


There are known bounds on the parameters $C, \lambda$ in 
Lemma~\ref{lemma:crate}, which characterize how large $C$ is and how far away
$\lambda$ is from $1$. Moreover, these bounds improve if the sequence $\{G(t)\}$ 
has some nice properties. The following lemma is a formal statement to this effect. 

\begin{lemma} \label{clambdachoices}
Let the graph sequence $\{G(t)\}$ be uniformly strongly connected.
Then, in the statement of Lemma~\ref{lemma:crate}(b) we have
\[ C = 2, \qquad \lambda = \left( 1 - \frac{1}{n^{nB}} \right)^{1/B}.\] 
If in addition every graph $G(t)$ is regular, we have
\[ C = \sqrt{2}, \   
\lambda = \min\left\{\left( 1 - \frac{1}{4n^3} \right)^{1/B}, \max_{t\ge0} \aor{\sigma_2(A(t))}\right\}. \] 
\end{lemma}

\aor{\begin{remark} We remark that this lemma is not novel relative to the previous literature, but rather is an explicit statement of the
bounds that have been implicitly used in the previous papers on the subject.  \end{remark}}

\begin{proof} 
From \cite{bhot, Morse, TBA}, under the assumption of the uniform strong connectivity of the graphs
and our definition of neighborhoods,  we have that:
if \[ x(t) = A'(t-1) \cdots A'(s) x(s) \quad \hbox{for all }t>s\ge0, \] 
then 
\begin{align*}
\max_{1\le i\le  n} x_i(t) - & \min_{1\le i\le n} x_i(t) 
\leq \left( 1 - \frac{1}{n^{nB}} \right)^{\lfloor (t-s)/(nB) \rfloor} \cr
& \times \left( \max_{1\le i\le n} x_i(s) - \min_{1\le i\le n} x_i(s) \right). \end{align*}
The preceding relation implies that for all $t>s\ge 0$,
\begin{align*}
\max_{1\le i \le n} x_i(t) - & \min_{1\le i\le n} x_i(t) 
\leq 2 \left( \left( 1 - \frac{1}{n^{nB}} \right)^{1/(nB)} \right)^{t-s} \cr
&\times \left( \max_{1\le i \le n} x_i(s) - \min_{1\le i\le n} x_i(s) \right). \end{align*}
This holds for every $x(s)$. 
By choosing $x(s)$ to be each of the $n$ basis vectors, we see that for every $j=1, \ldots, n$, 
\begin{align*}
&\max_{i} [A'(t)\cdots A'(s)]_{ij} - \min_i [A'(t) \cdots A'(s)]_{ij} \cr
&\leq 2 \left( \left( 1 - \frac{1}{n^{nB}} \right)^{1/(nB)} \right)^{t-s}. \end{align*}
Since each matrix $A'(t)$ is row-stochastic, the 
entry $\phi_j(s)$ is a limit of the convex combinations of the $n$ numbers $[A'(t) \cdots A'(s)]_{ij}, i = 1, \ldots, n$, as $t\to\infty$. 
Hence, we have proven the first relation of the lemma
for $t>s$. For $t=s$, since the matrix $A(s)$ is column-stochastic and $\phi(s)$ is a stochastic vector,
we obviously have $|[A'(s)]_{ij}-\phi_j(s)|\le 2$ for all $i,j$, 
showing that the relation also holds for $t=s$.

As for the second statement, when the graphs $G(t)$ are regular, 
each of the matrices $A(t)$ is doubly stochastic. Then, 
the results of $\cite{noot09}$ imply that:
if \[ x(t) = A'(t-1) \cdots A'(s) x(s)\quad\hbox{for all }t>s\ge0, \] 
then for all $t>s\ge 0$ we have
\[ \| x(t) - \bar{x}_s\1 \|^2 
\leq \left( 1 - \frac{1}{2n^3} \right)^{\lfloor (t-s)/B \rfloor} 
\|x(s) - \bar{x}_s\1 \|^2, \]
where $\bar{x}_s$ is the average of the entries of $x(s)$. 
From the preceding inequality, we can see that for all $t>s\ge 0$,
\[ \|x(t) - \bar{x}_s\1 \|^2 \leq 2 \left( 1 - \frac{1}{2n^3} \right)^{(t-s)/B}
   \| x(s) - \bar{x}_s\1 \|^2, \]
implying that 
\[\max_{i} |x_i(s)-\bar x_s| \le \sqrt{2  \left(1 - \frac{1}{2n^3}\right)^{(t-s)/B} }
\| x(s) - \bar{x}_s\1 \|.\]  
Moreover, since the relation above holds for any vector $x(s)\in\R^n$, by 
plugging in each basis vector $e_j\in\R^n$ and noting that $\|e_j-(1/n)\1\|\le 1$, 
we obtain for each $j$ and all $t>s\ge0$,
\[ \max_{1\le i\le n} [A'(t) \cdots A'(s)]_{ij} - \frac{1}{n} 
\leq \sqrt{2} \left(\sqrt{  1 - \frac{1}{2n^3}}\right)^{(t-s)/B}.  \] 
Since $\sqrt{1 - \beta/2} \leq 1-\beta/4$ for all $\beta \in (0,1)$, it follows that we may choose 
$C=\sqrt{2}$  and $\lambda = (1 - 1/(4n^3))^{1/B}$.
The same line of argument can be used to show that we may choose 
$C=1$ and $\lambda = \max_{t\ge0} \aor{\sigma_2(A(t))}$.
\end{proof}

\aor{We will end up applying this lemma to a sequence of graphs which is only uniformly strongly connected after throwing out 
a few graphs at the start. To that end, we have the following corollary whose proof is straightforward.} 

\aor{\begin{corollary} \label{almoststrongly} 
Suppose that the graph sequence $G(t)$ has the following property: 
there exists an integer $T>0$ such that  given any integer $s\ge0$, 
there is a time $0\le t_s \leq T$  for which the
graph sequence $G(s+t_s), G(s+t_s+1), G(s+t_s+2), \ldots$ is uniformly strongly connected. 
Then, for each integer $s \geq 0$ there is a stochastic vector
$\phi(s)$ such that for all $i,j$ and $t \geq s$, 
\[ \left| [A'(t) \cdots A'(s)]_{ij} - \phi_j(s) \right| \leq C \lambda^{t- s - T}, \] 
where $C, \lambda$ may be taken as in Lemma \ref{clambdachoices}. 
\end{corollary}}

In the proof of Lemma~\ref{lemma:newx}, we will make use of 
a lower bound on the entries of the vectors $\1'A(t) \cdots A(0)$. 
To provide such a bound, we employ the following  intermediate result
which provides a uniform lower bound on the entries of the vectors $\1'A'(t) \cdots A'(0)$. 

\begin{lemma} \label{lemma:delta}
Given a graph sequence $\{G(t)\}$, define 
\begin{align*}
 \delta' \triangleq \inf_{t=0,1, \ldots} \left(\min_{1\le i \le n}  [\1' A'(t) \cdots A'(0)]_i \right). \end{align*}
If the graph sequence $\{G(t)\}$ is uniformly strongly connected, then  $\delta' \geq \frac{1}{n^{nB}}.$
If each $G(t)$ is regular, then $\delta' = 1$.
\end{lemma} 

\begin{proof} 
By the definition of matrices $A(t)$ in Eq.~\eqref{eq:defA}, we have $[A(t)]_{ii}=1/d_i(t)$. Since $d_i(t)\le n$,
it follows that $[A(t)]_{ii} \aor{\geq} 1/n$ for all $t$ and $i$.
Therefore, for all $i$,
\[  [A'(t+1) \cdots A'(0) ]_{ii} \geq \frac{1}{n}  [A'(t) \cdots A'(0)]_{ii}
\ \hbox{for $t\ge0$}.\] 
Thus, we certainly have  $[\1'A'(t) \cdots A'(1)]_i \geq 1/n^{nB}$ for all $i$ and all $t$ in 
the range $1 \leq t \leq n^{nB}$. 
However, it was 
shown in \cite{Morse, TBA} that for $t > (n-1)B$, every entry of $A'(t) \cdots A'(1)$ is positive and has value 
at least $1/n^{nB}$. 
Since $n^{nB} > (n-1)B$, this proves the bound $\delta' \geq 1/n^{nB}$. The final claim that 
$\delta'=1$ for a sequence of regular graphs is trivial. 
\end{proof} 

\begin{remark}\label{remark:phi} 
As an immediate consequence of the definition of $\delta'$, 
we have $\phi_j(s) \geq \delta'/n$ for all $j$. 
\end{remark}

By taking transposes and applying Lemmas \ref{lemma:crate}, \ref{clambdachoices}, \ref{lemma:delta}, 
 we immediately obtain the
 following result on the products $A(t) \cdots A(s)$. 
 For convenience, let us adopt the notation of denoting these products
 by $A(t:s)$, i.e.,
 \[A(t:s) \triangleq A(t) \cdots A(s)\quad\hbox{for all }t\ge s\ge0.\]

\begin{corollary} \label{rowstochasticcorr} 
Let the graph sequence $\{G(t)\}$ be uniformly strongly connected.  
Then, the following statements are valid:
\begin{itemize}
\item[(a)]  There is a sequence $\{\phi(t)\}$ of stochastic vectors $\phi(t)\in\R^n$ 
such that the matrix difference 
$\mbox{$A(t:s)$} - \phi(t) \1'$ for $t\ge s$ 
decays geometrically, i.e., for all $i,j=1,\ldots,n,$
\[|[A(t:s)]_{ij}-\phi_i(t)|\le C\lambda^{t-s}\quad\hbox{for all }t\ge s\ge0,\] 
where we can always choose 
$$C=4, \quad \lambda = \left( 1 - 1/n^{nB} \right)^{1/B}.$$
If in addition each $G(t)$ is regular, 
we may choose 
$$C=2\sqrt{2}, \quad \lambda = \left( 1 - 1/(4n^3) \right)^{1/B},$$ or 
$$C=\aor{\sqrt{2}}, \quad \lambda = \max_{t\ge0} \aor{\sigma_2(A(t))},$$
whenever $\sup_{t\ge0} \aor{\sigma_2(A(t))}<1$. 
\item[$($b)] 
The quantity 
\[ \delta = \inf_{t=0,1, \ldots} \left(\min_{1\le i\le n}  [A(t) \cdots A(0) \aor{\1}]_i \right) \]
satisfies \[ \delta \geq \frac{1}{n^{nB}}.\] 
Moreover, if the graphs $G(t)$ are regular, we have $\delta = 1$. 
\item[(c$)$]  The stochastic vectors $\phi(t)$ satisfy for all $j,$
\[ \phi_j(t) \geq  \frac{\delta}{n}\qquad\hbox{for all times $t\ge0$}.\]   
\end{itemize}
\end{corollary}
\begin{proof} The results follow by employing \aor{Corollary \ref{almoststrongly}}, Lemmas \ref{clambdachoices}, \ref{lemma:delta}, and Remark \ref{remark:phi}, where we take the transposes of the matrices. The only
aspect that needs to be verified is that these lemmas can be applied, which is routine, 
with the exception of the issue of the uniform strong connectivity requiring elaboration.

When we transpose $A(t:s)$ to obtain the
product $A'(s) \cdots A'(t-1) A'(t)$, we have reversed the order in which the matrices appear.
Moreover, by taking the transposes of each matrix, we have effectively reversed the direction of 
every edge in each graph $G(t)$. 
\aor{In the case when $\sup_{t \geq 0} \sigma_2(A(t)) < 1$, it is easy
to see  that $B=1$ and the resulting ``reversed'' sequence  is connected at every step, 
so the bound of Lemma \ref{clambdachoices} applies verbatim. In the 
other cases,   the resulting sequence is still strongly connected
once we throw out at most $B$ graphs from the start of the sequence.
The bounds of Corollary \ref{almoststrongly} thus apply} 
with $t-s$ replaced by $t-s-B$, which we take care of by instead doubling the constant $C$. 
\end{proof} 

The parameter $\delta$ as defined in Lemma \ref{lemma:delta}   
may be thought of as a measure of imbalance of the influence among the nodes.  Indeed, 
$\delta$ is defined as the best lower bound on the row sums of the matrices $A(t:s)$. In the case when each of the graphs $G(t)$ is regular, the matrices $A(t)$ will be doubly
stochastic and we will have $\delta=1$ as previously remarked (i.e., no imbalance of influence). 
By contrast, when $\delta \approx 0$, there is a node $i$ 
such that the $i$'th row in some $A(t:s)$ will have entries which are all nearly zero; in short, it is almost as if node $i$ has no in-neighbors at all.


We proceed with our sequence of intermediate lemmas for the proof of Lemma \ref{lemma:newx}. 
We will now need some auxiliary results on the convolution
of two scalar sequences, as in the following statement.
\begin{lemma} 
  \label{lemma:betagama}
  {\rm  (\cite{Ram2010} Lemma 3.1)}
Let $\{\gamma_k\}$ be a scalar sequence.
\begin{itemize}
\item [(a)] If $\lim_{k\to\infty}\gamma_k=\gamma$  and $0< \beta<1,$ then
$\lim_{k\to\infty}\sum_{\ell=0}^k \beta^{k-\ell}\gamma_\ell
=\frac{\gamma}{1-\beta}.$
\item [(b)] If $\gamma_k\ge0$ for all $k$,
 $\sum_{k=0}^\infty\gamma_k<\infty$ and $0< \beta<1,$  then
  $\sum_{k=0}^\infty
  \left(\sum_{\ell=0}^k\beta^{k-\ell}\gamma_\ell\right)<~\infty.$
\end{itemize}
\end{lemma}

With these pieces in places, we can now proceed to the proof of Lemma \ref{lemma:newx}. Our argument will 
rely on Corollary~\ref{rowstochasticcorr} on the products $A(t:s)$ and the just-stated 
Lemma~\ref{lemma:betagama} on the convolution sequences. 

\begin{proof}[Proof of Lemma \ref{lemma:newx}]
\noindent
(a) \ 
By inspecting Eq.~(\ref{eq:pert}) it is easy to see that for all $t\ge0$,
\begin{equation}\label{eq:xt}
x(t+1) =A(t:0)x(0) + \sum_{s=1}^{t}A(t:s)\e(s) + \e(t+1),
\end{equation}
which implies
\begin{equation} \label{eq:ax} 
A(t+1)x(t+1) = A(t+1:0)x(0) + \sum_{s=1}^{t+1}A(t+1:s)\e(s).\end{equation} 
Moreover, since each
$A(t)$ is column-stochastic, we have that $\1' A(t) = \1'$ and Eq.~(\ref{eq:xt}) 
further implies that
\begin{equation} \label{eq:xsum}
\1'x(t+1)  = \1'x(0) + \sum_{s=1}^{t+1}\1'\e(s)
\qquad\hbox{for all }t\ge0. \end{equation}
Now, from~Eq.~(\ref{eq:ax}) and Eq.~(\ref{eq:xsum}) we obtain for all $t\ge0$,
\begin{eqnarray}\label{eq:oo2}
&& A(t+1)x(t+1) - \phi(t+1)\1'x(t+1) \cr
&& =  \left( A(t+1:0)- \phi(t+1)\1' \right) x(0) \cr
&& \ \ + \sum_{s=1}^{t+1}\left( A(t+1:s)-\phi(t+1)\1' \right)\e(s).
\end{eqnarray}
According to Corollary~\ref{rowstochasticcorr}, if we define $D(t,s)$ to be 
\[ D(t,s) = A(t:s)- \phi(t) \1'\]
then we have the entry-wise decay bound
\begin{align}\label{eq:dnerr}
|[D(t,s)]_{ij}|\le C\lambda^{t-s} \ \hbox{for all $i,j$ and $t\ge s\ge 0$},\end{align}
where the constants $C>0$ and $\lambda \in (0,1)$ have the properties listed in 
Corollary~\ref{rowstochasticcorr}. 

 Therefore, from relation~\eqref{eq:oo2} it follows for $t\ge0$,
\begin{align*}
&A(t+1)x(t+1) =  \phi(t+1)\1'x(t+1) \cr
& + D(t+1,0) x(0) + \sum_{s=1}^{t+1} D(t+1,s) \e(s).\end{align*}
Thus, for $t\ge1$ we have
\begin{align}\label{eq:oo3}
w(t+1) =
& A(t)x(t) =\phi(t)\1'x(t) \cr
& + D(t,0) x(0) + \sum_{s=1}^{t} D(t,s) \e(s).
\end{align}
We may derive a similar expression for $y(t+1)$: 
\begin{align}\label{eq:oo4}
y(t+1) &= A(t:0)y(0)\cr
&=\phi(t)\1'y(0)+ D(t,0)y(0)\cr
& =\phi(t) n+ D(t:0)\1,
\end{align}
which holds for all $t\ge0$.
From~\eqref{eq:oo3} and~\eqref{eq:oo4} we obtain for every $t\ge1$ and all~$i$,
\begin{align*}
& z_i(t+1)=\frac{w_i(t+1)}{y_i(t+1)}\cr
& =\frac{\phi_i(t)\,\1'x(t) + [D(t:0) x(0)]_i + \sum_{s=1}^{t} [D(t:s) \e(s)]_i}
{\phi_i(t)\,n+ [D(t:0)\1]_i}.\end{align*}
Therefore,
\begin{align*}
& z_i(t+1) -\frac{\1'x(t)}{n} \cr 
&=\frac{\phi_i(t)\,\1'x(t) + [D(t:0) x(0)]_i + \sum_{s=1}^{t} [D(t:s) \e(s)]_i}
{\phi_i(t)\,n+ [D(t:0)\1]_i} \cr
&\ \ -\frac{\1'x(t)}{n}.\end{align*}
By bringing the fractions to a common denominator, after the cancellation of some terms, we find that
\begin{align*}
&z_i(t+1) -\frac{\1'x(t)}{n} \cr
&= \frac{n[D(t:0) x(0)]_i + n\sum_{s=1}^{t} [D(t:s) \e(s)]_i}  {n\left(\phi_i(t)\,n+ [D(t:0)\1]_i \right)}\cr
& \ \ - \frac{\1'x(t) [D(t:0)\1]_i}{n\left(\phi_i(t)\,n+ [D(t:0)\1]_i \right)}.
\end{align*} 
Observe that the denominator of the above fraction is $n$ times the $i$'th row sum of $A(t:0)$. By definition of $\delta$, this row sum is at least $\delta$, and consequently 
\[\phi_i(t)\,n+ [D(t:0)\1]_i = [A(t:0)\1]_i\ge \delta. \] 
Thus, for all $i$ and $t\ge1$,
\begin{eqnarray*}
&&\left|z_i(t+1) - \frac{\1' x(t)}{n}\right|\cr
&&\hbox{}\cr
&&\le \frac{\left| [D(t:0) x(0)]_i + \sum_{s=1}^{t} [D(t:s) \e(s)]_i\right|} 
{\phi_i(t)\,n+ [D(t:0)\1]_i }\cr
&&\hbox{}\cr
&&\ \ +\frac{\left| \1'x(t) [D(t:0)\1]_i\right|}{n\left(\phi_i(t)\,n+ [D(t:0)\1]_i \right)}\cr
&&\le \frac{1}{\delta}\,\left( \left(\max_{j}|[D(t:0)]_{ij}| \right)\|x(0)\|_1 \right.\cr
&& \ + \left. 
\sum_{s=1}^{t}\left(\max_{j}| [D(t:s)]_{ij} |\right)\|\e(s)\|_1 \right)\cr
&& \ 
+\frac{1}{n\delta}|\1'x(t)|\left(\max_{j}|[D(t:0)]_{ij}|\right) n.
\end{eqnarray*}
Factoring $n$ out in the last term,
and using estimates for $| [D(t:s)]_{ij} |$ as given in~\eqref{eq:dnerr}, we obtain 
\begin{eqnarray*}
&&\left|z_i(t+1) - \frac{\1' x(t)}{n}\right|
\le \frac{C}{\delta}\lambda^t\|x(0)\|_1\cr
&& \ + \frac{C}{\delta}\left(\sum_{s=1}^{t}\lambda^{t-s}\|\e(s)\|_1 + |\1'x(t)|\lambda^t\right).
\end{eqnarray*}
Now we look at the term $\1'x(t)$. From Eq. (\ref{eq:xsum}), we have
\[|\1'x(t)| \le \|x(0)\|_1 +  \sum_{s=1}^t \|\e(s)\|_1.\]
From the preceding two relations it follows that for all $i$ and $t\ge1$,  
\begin{eqnarray*}
&&\left|z_i(t+1) - \frac{\1' x(t)}{n}\right|
\le \frac{C}{\delta}\,\lambda^t\|x(0)\|_1\cr
&& \ +\frac{C}{\delta} \sum_{s=1}^{t}\lambda^{t-s}\|\e(s)\|_1 \cr
&& \ + \frac{C}{\delta} \ \lambda^t\left( \|x(0)\|_1 +  \sum_{s= 1}^t \|\e(s)\|_1\right)\cr
&=&
\frac{C}{\delta}\,\left(2\lambda^t\|x(0)\|_1 + 2 \sum_{s=1}^{t}\lambda^{t-s}\|\e(s)\|_1 \right).
\end{eqnarray*} Since we were able to choose $C \leq 4$ in all the cases considered in Lemma \ref{rowstochasticcorr}, we may choose $C=4$ to obtain
the result in part (a).

\noindent
(b) \ By letting $t\to\infty$ in the preceding relation, since $\lambda\in(0,1)$, we find that
for all $i,$
\[\lim_{t\to\infty} \left|z_i(t+1) - \frac{\1' x(t)}{n}\right|
\le \lim_{t\to\infty}\sum_{s=1}^{t}\lambda^{t-s}\|\e(s)\|_1.\]
When $\e_i(t)\to0$ for all $i$, then $\|\e(t)\|_1\to 0$ and, by Lemma~\ref{lemma:betagama}(a),
we conclude that 
\[\lim_{t\to\infty}\sum_{s=1}^{t}\lambda^{t-s}\|\e(s)\|_1=0,\]
and the result follows from the preceding two relations.

\noindent
$($c) \
Since $\{\a(t\}$ is positive and non-increasing sequence, we have $\a(t+1)\le \a(1)$ for all $t\ge0$ 
and $\a(t+1)\le \a(s)$ for  all $t\ge s\ge 0$. Using these relations, we obtain
\begin{align}\label{eq:alphazbound}
&\ \ \a(t+1)\left|z_i(t+1) - \frac{\1' x(t)}{n}\right|\cr
\le &
\frac{2C}{\delta}\,\left(\a(1)\lambda^t\|x(0)\|_1 +  \sum_{s=1}^{t}\lambda^{t-s}\a(s)\|\e(s)\|_1 \right). 
\quad
\end{align}
Since $\lambda\in (0,1)$, the sum $\sum_{t=1}^\infty \lambda^t$ is finite, and by Lemma~\ref{lemma:betagama}(b)
the sum $\sum_{t=1}^\infty \sum_{\ell=1}^t \lambda^{t-\ell} \a(\ell)\|\e(\ell)\|_1$ is finite. Therefore
\[\sum_{t=0}^\infty \a(t+1)\left|z_i(t+1) - \frac{\1' x(t)}{n}\right|<\infty.\]
\end{proof}

Lemma \ref{lemma:newx}, which we have just proved, is the central result of this section.
It states that each of the sequences $z_i(t+1)$ tracks the average $\bar x(t) = \1'x(t)/n$ increasingly well as time goes on. We will later require a corollary of this lemma showing that a weighted 
time-average of each $z_i(t+1)$ tracks a weighted time-average of $\bar x(t)$.
The next corollary gives a precise statement of this result. 
In the derivation, we also use the following inequality
\begin{equation}\label{eq:sqrtsum} 
\sum_{k=0}^t \frac{1}{\sqrt{k+1}} \ge \sqrt{t+1}\qquad\hbox{for all }t\ge1,  
\end{equation} 
which follows by
\[\sum_{k=0}^t \frac{1}{\sqrt{k+1}} \ge \int_0^{t+1} \frac{du}{\sqrt{u+1}}
= 2 \left( \sqrt{t+2} - 1 \right)\]
and the relation $2 \left( \sqrt{t+2} - 1 \right)\ge \sqrt{t+1}$ for all $t\ge1$.

\begin{corollary} \label{cor:sqrtavg} 
Suppose that all the assumptions of Lemma \ref{lemma:newx} are satisfied. Moreover,
let $\alpha(t) = 1/\sqrt{t}$ for all $t\ge1$ and the perturbations $\e_i(t)$ are bounded as follows:
\[\|\e(t)\|_1 \leq D/\sqrt{t}\qquad\hbox{for all }t\ge1,\] 
for some scalar $D>0.$
Defining 
\[ \bar x(t) = \frac{1}{n} \sum_{i=1}^n x_i(t) \qquad\hbox{for all }t\ge0,\] 
we have for every $i=1, \ldots, n$ and $t\ge1$, 
\begin{align} \label{eq:zxdiff}
&\sum_{k=0}^t \alpha(k+1) \left| z_i(k+1) -\bar x (k) \right| \cr
& \leq \frac{8}{\delta(1-\lambda)} \left(\|x(0)\|_1 + D(1+ \ln t)\right),\end{align}
\begin{align}\label{eq:averdiff}
&\frac{1}{\sum_{k=0}^t \alpha(k+1)} \sum_{k=0}^t \alpha(k+1)|z_i(k+1) - \bar{x}(k)|\cr
&\leq 8 \frac{\|x(0)\|_1 + D (1 + \ln t)}{\delta(1-\lambda) \sqrt{t+1} },
\end{align}
where $\delta\in (0,1]$ and $\lambda\in (0,1)$ are as given in Lemma~\ref{lemma:newx}.
\end{corollary}

\begin{proof} 
From Lemma~\ref{lemma:newx}(a) we have for all $i$ and all $t\ge1$, 
\begin{eqnarray*}
&&\sum_{k=1}^t \alpha(k+1) \left| z_i(k+1) -\bar x (k) \right| \cr
& \leq & \frac{8}{\delta} \sum_{k=1}^t \frac{\lambda^k}{\sqrt{k+1}} \|x(0)\|_1  \cr
&& + \frac{8}{\delta} \sum_{k=1}^t \alpha(k+1) \sum_{s=1}^k \lambda^{k-s} \|\e(s)\|_1 \cr
& \leq & \frac{8}{\delta} \frac{\lambda}{1-\lambda} \|x(0)\|_1 
+  \frac{8D}{\delta} \sum_{k=1}^t \sum_{s=1}^k \frac{\lambda^{k-s}}{s}.
\end{eqnarray*}
Note that 
\[\sum_{k=1}^t \sum_{s=1}^k \frac{\lambda^{k-s}}{s}
=\sum_{s=1}^t \frac{1}{s}\sum_{k=s}^t 
\lambda^{k-s}\le \sum_{s=1}^t \frac{1}{s}\frac{1}{1-\lambda}.\]
Furthermore, since
\[\sum_{s=1}^t \frac{1}{s}= 1 + \sum_{s=2}^t \frac{1}{s}\le 1+ \int_1^{t}\frac{du}{u}=1+ \ln t,\]
it follows that 
\begin{eqnarray} \label{eq:1}
&&\sum_{k=1}^t \alpha(k+1) \left| z_i(k+1) -\bar x (k) \right| \cr
& \leq & \frac{8\lambda}{\delta(1-\lambda)} \|x(0)\|_1 + \frac{8D}{\delta} \frac{ (1+\ln t)}{1-\lambda}.\end{eqnarray}
Note that for $k=0$, we have
\[\alpha(1) \left| z_i(1) -\bar x (0) \right|\le |z_i(1)|+|\bar x(0)|.\]
From the definition of the perturbed push-sum method, we have
\[z_i(1)=\frac{w_i(1)}{y_i(1)}, \]
\[ w_i(1)= \sum_{j\in N^{\rm in}_i(0)} \frac{x_j(0)}{d_j(0)}, 
\qquad y_i(1)=\sum_{j\in N^{\rm in}_i(0)}\frac{1}{d_j(0)},\]
where the last equality follows from $y_i(0)=1$ for all $i$. 
Thus, \aor{$z_i(1)$} is a weighted average of the entries in $x(0)$, while $\bar x(0)$ is the average of these entries.
Hence,
\[|z_i(1)|+|\bar x(0)|\le 2\max_{i} |x_i(0)|\le 2\|x(0)\|_1.\]  
Since $\delta\le1$, we see that
\begin{equation}\label{eq:2}
 \alpha(1) \left| z_i(1) -\bar x (0) \right|\le \frac{2}{\delta}\|x(0)\|_1< \frac{8}{\delta}\|x(0)\|_1.
 \end{equation}
By summing the relations in Eqs.~\eqref{eq:1} and~\eqref{eq:2}, we obtain
the estimate in Eq.~\eqref{eq:zxdiff}.
The relation in Eq.~\eqref{eq:averdiff} follows from $\a(t)=1/\sqrt{t}$ and
Eqs.~\eqref{eq:sqrtsum}--\eqref{eq:zxdiff}.
\end{proof}

We conclude this section by noting that Lemma~\ref{lemma:newx} and 
Corollary~\ref{cor:sqrtavg} \aor{can be extended to the case when} $x_i(t)$ (and, by extension, $z_i(t)$) is a $d$-dimensional vector,
by applying the results to each coordinate component of the space.

\section{Convergence Results for Subgradient-Push Method}\label{sec:optim}
We turn now to the proofs of our main results,  namely Theorems~\ref{mainthm} and~\ref{convrate}.
Our arguments will crucially rely on the convergence results for the perturbed push-sum method we have established in Section~\ref{sec:averaging}. 

We give a brief, informal summary of the main ideas behind our argument. The convergence result for 
the perturbed push-sum method of Section~\ref{sec:averaging} implies that, under the appropriate assumptions, the entries of $\bz_i(t)$ get close to each other over time, 
and consequently $\bz_i(t)$ approaches
a multiple of the all-ones vector. Thus, every node takes a subgradient of its own function 
$f_i$ at nearly the same point.
Over time, these subgradients are averaged over the nodes 
by the push-sum-like updates of our method. Consequently, the subgradient-push method 
approximates the ordinary (centralized) subgradient algorithm applied to 
the average function  $\frac{1}{n} \sum_{j=1}^n f_j$.

We now begin the formal process of proving Theorems~\ref{mainthm} and~\ref{convrate}. 
In what follows, we will use a deterministic counterpart of the well-known (almost) supermartingale convergence result
(\cite{robbins1971}; see also~\cite{polyak87}, Lemma 11, Chapter~2.2). 
The result is given in the following lemma.

\begin{lemma}\label{lemma:polyak}
Let $\{v_t\}$ be a non-negative scalar sequence such that 
\[v_{t+1}\le (1+b_t) v_t - u_t +c_t\qquad\hbox{for all }t\ge0,\]
where $b_t\ge0,$ $u_t\ge0$ and $c_t\ge0$ for all $t\ge0$ with $\sum_{t=0}^\infty b_t<\infty$, and 
$\sum_{t=0}^\infty c_t<\infty$. 
Then, the sequence $\{v_t\}$ converges to some $v\ge0$ and 
$\sum_{t=0}^\infty u_t<\infty$.
\end{lemma}

Our first step is to establish a lemma pertinent to the convergence of a 
sequence satisfying a subgradient-like recursion. In the proof of this lemma,
we make use of Lemma~\ref{lemma:polyak}. 

\begin{lemma}\label{lemma:opt}
Consider a convex minimization problem $\min_{x \in \R^m} f(x)$, where $f:\mathbb{R}^m\to\mathbb{R}$ is a \aor{continuous} function. Assume that 
the solution set $X^*$ of the problem is nonempty. Let $\{x_t\}$ be a sequence
such that for all $x\in  \aor{X^*}$ and for all $t\ge0,$
\[\|x_{t+1}-x\|^2 \le (1+b_t)\|x_t-x\|^2 - \a_t\left (f(x_t)-f(x)\right) +c_t,\]
where $b_t\ge0,$ $\a_t\ge0$ and $c_t\ge0$ for all $t\ge0$, with $\sum_{t=0}^\infty b_t<\infty$, 
$\sum_{t=0}^\infty \a_t=\infty$ and $\sum_{t=0}^\infty c_t<\infty$. 
Then, the sequence $\{x_t\}$ converges to some solution $x^*\in X^*$. 
\end{lemma}
\begin{proof}
By letting $x=x^*$ for arbitrary $x^*\in X^*$ and by defining $f^*=\min_{x\in \aor{\R^m}} f(x)$, we obtain
for all $t\ge0$,
\[\|x_{t+1}-x^*\|^2 \le (1+b_t)\|x_t-x^*\|^2 - \a_t\left (f(x_t)-f^*\right) +c_t.\] 
Thus, all the conditions of Lemma~\ref{lemma:polyak} are satisfied, and by this lemma
we obtain the following statements:
\begin{equation}\label{eq:xtconv}
\hbox{$\{\|x_t - x^*\|^2\}$ converges for each $x^*\in X^*$},
\end{equation}
\begin{equation}\label{eq:sumf}
\sum_{t=0}^\infty \a_t \left(f(x_t) - f^* \right)<\infty.\end{equation}
Since $\sum_{t=0}^\infty \a_t=\infty$, it follows from~\eqref{eq:sumf} that 
\[\liminf_{t\to\infty} f(x_t)=f^*.\] Let $\{x_{t_\ell}\}$ be a subsequence of $\{x_t\}$ such that
\begin{equation}\label{eq:lim}
\lim_{\ell\to\infty} f(x_{t_\ell})=\liminf_{t\to\infty} f(x_t)=f^*.\end{equation} 
Now, Eq. (\ref{eq:xtconv}) implies that the sequence $\{x_t\}$ is bounded.
 Thus, without loss of generality, we can assume that $\{x_{t_\ell}\}$ is converging to some 
 $\tilde x$ (for otherwise, we
can in turn select a convergent subsequence of  $\{x_{t_\ell}\}$). Therefore,
\[\lim_{\ell\to\infty} f(x_{t_\ell})=f(\tilde x),\]
which by Eq.~\eqref{eq:lim} implies that $\tilde x\in X^*$. By letting $x^*=\tilde x$ in  
Eq.~\eqref{eq:xtconv}
we obtain that $\{x_t\}$ converges to $\tilde x$.
\end{proof}

A key relations in the proofs of Theorems~\ref{mainthm} and~\ref{convrate} will be obtained 
by applying Lemma~\ref{lemma:opt} to the average process 
$$ \bar \bx(t) = \frac{1}{n}\sum_{i=1}^n \bx_i(t)\qquad\hbox{for }t\ge0,$$ 
where $\bx_i(t)$ is the sequence
generated by the subgradient-push method. We will need to
argue that $\bar \bx(t)$ satisfies the assumptions of Lemma~\ref{lemma:opt}, 
for which the following result will be instrumental.
 
\begin{lemma}\label{lemma:key} 
Under the same assumptions as in Theorem \ref{mainthm}, we
have for all $\bv\in\mathbb{R}^d$ and $t\ge0$,
\begin{eqnarray*}
&&\|\bar \bx(t+1) - \bv\|^2 
\le   \|\bar \bx(t) - \bv\|^2 \cr
&& \ - \frac{2\alpha(t+1)}{n} \left(F(\bar\bx(t)) - F(\bv) \right) \cr
&& \  + \frac{4\alpha(t+1)}{n} \sum_{i=1}^n L_i \|\bz_i(t+1)-\bar \bx(t)\|\cr
&&  \ +\alpha^2(t+1)\frac{\aor{L^2}}{n^2}.\end{eqnarray*}
\end{lemma}
\begin{proof} 
Let us define $\widetilde{x}_\ell (t)$ to be the vector in $\R^n$ which stacks up 
the $\ell$'th entries of all the vectors $\bx_i(t)$: formally, 
we define $\widetilde{x}_\ell(t)$ to be the vector whose $j$'th entry is the $\ell$'th entry of $\bx_j(t)$. 
Similarly, we define $\widetilde{g_\ell}(t)$ to be
the vector stacking up the $\ell$'th entries of the vectors $\bg_i(t)$: 
the $j$'th entry of $\widetilde{g_\ell}(t)$ is the $\ell$'th entry of $\bg_j(t)$. 

From the definition of the subgradient-push in Eq. (\ref{eq:minmet}) it can be seen that 
for $\ell=1,\ldots,d$,
\[ \widetilde{x}_\ell(t+1) = A(t) \widetilde{x}_\ell(t) - \alpha(t+1) \widetilde{g}_\ell(t+1).\] 
Since $A(t)$ is a column-stochastic matrix, it follows that for all $\ell=1,\ldots,d$,
\[ \frac{1}{n} \sum_{j=1}^n [\widetilde{x}_\ell(t+1)]_j
 =  \frac{1}{n} \sum_{j=1}^n [\widetilde{x}_\ell(t)]_j 
 - \frac{\alpha(t+1)}{n} \sum_{j=1}^n [\widetilde{g}_\ell(t+1)]_j.\]
Since the $\ell$'th entry of $\bar \bx(t+1)$ is exactly the left-hand side above, we can conclude that
\begin{equation} \label{eq:avdone} 
\bar \bx(t+1) = \bar \bx(t) - \frac{\alpha(t+1)}{n} \sum_{j=1}^n \bg_j(t+1). 
\end{equation}

Now let $\bv\in\mathbb{R}^d$ be an arbitrary vector. 
From relation~\eqref{eq:avdone} it follows that
for all $t\ge0$,
\begin{align*}
\|\bar\bx(t+1) - \bv\|^2 = & \|\bar\bx(t) - \bv\|^2 \cr
& - \frac{2\alpha(t+1)}{n} \sum_{i=1}^n \bg'_i(t+1)(\bar\bx(t) - \bv) \cr
& +\frac{\alpha^2(t+1)}{n^2}\left\|\sum_{i=1}^n \bg_i(t+1)\right\|^2.\end{align*}
Since the subgradient norms of each $f_i$ are uniformly bounded by $L_i$, it further 
follows that for all $t\ge0$,
\begin{align}\label{eq:beg1}
\|\bar\bx(t+1)  - &\bv\|^2 
\le \|\bar\bx(t) - \bv \|^2 \cr
& - \frac{2\alpha(t+1)}{n} \sum_{i=1}^n \bg'_i(t+1)(\bar\bx(t) - \bv) \cr
& +\alpha^2(t+1)\frac{\aor{L^2}}{n^2}.\end{align}

We next consider a cross-term $\bg'_i(t+1)(\bar\bx(t) - \bv)$ in~\eqref{eq:beg1}.  
For this term, we write
\begin{align}\label{eq:beg2}
 \bg_i'(t+1)(\bar\bx(t) - \bv)
=&   \bg_i'(t+1)(\bar\bx(t) - \bz_i(t+1) \aor{)}\cr
& + \bg_i'(t+1)(\bz_i(t+1) - \bv).
\end{align}
Using the subgradient boundedness and the Cauchy-Schwarz inequality, we can lower bound the first term
$\bg_i'(t+1)(\bar\bx(t) - \bz_i(t+1))$ as 
\begin{equation}\label{eq:beg3}
\bg_i'(t+1)(\bar\bx(t) - \bz_i(t+1))
\ge -L_i \|\bar\bx(t) - \bz_i(t+1)\|.
\end{equation} 
As for the second term $\bg_i'(t+1) (\bz_i(t+1) - \bv)$, we can 
use the fact that $\bg_i'(t+1)$ is the subgradient of $f_i(\theta)$ at $\theta=\bz_i(t+1)$ 
to obtain: 
\[\bg'_i(t+1)(\bz_i(t+1) - \bv) \ge  f_i(\bz_i(t+1)) - f_i(\bv),\]
from which, by adding and subtracting \aor{$f_i\left(\bar\bx(t)\right)$} and using the 
Lipschitz continuity of $f_i$ (implied by the subgradient boundedness), we further obtain
\begin{eqnarray}\label{eq:beg4}
\bg'_i(t+1) (\bz_i(t+1)-  \bv) 
&\ge &  - L_i\|\bz_i(t+1)-\bar\bx(t)\| \cr
&&  + f_i(\bar\bx(t)) - f_i(\bv).
\end{eqnarray}
By substituting the estimates of Eqs.~\eqref{eq:beg3}--\eqref{eq:beg4} 
back in relation~\eqref{eq:beg2}, and using $F(\bx)= \sum_{i=1}^n f_i(\bx)$
we obtain
\begin{align}\label{eq:beg5}
\sum_{i=1}^n \bg'_i(t+1) & (\bar\bx(t)  -  \bv)
\ge  F(\bar\bx(t)) - F(\bv)  \cr
& -  2\sum_{i=1}^n L_i \|\bz_i(t+1)-\bar\bx(t)\|.  
\end{align}
Now, we substitute estimate~\eqref{eq:beg5} into relation~\eqref{eq:beg1} and obtain for
any $\bv\in\mathbb{R}^d$ and all $t\ge0$,
\begin{align*}
\|\bar\bx(t+1) & - \bv \|^2 
\le  \|\bar\bx(t) - \bv \|^2 \cr
& - \frac{2\alpha(t+1)}{n} \left(F(\bar\bx(t)) - F(\bv) \right) \cr
& + \frac{4\alpha(t+1)}{n} \sum_{i=1}^n L_i \|\bz_i(t+1)-\bar\bx(t)\| \cr
& +\alpha^2(t+1)\frac{\aor{L^2}}{n^2}.\end{align*}
\end{proof}

With all the pieces in place, we are finally ready to prove Theorem \ref{mainthm}. 
The proof idea is to show that the averages $\bar\bx(t)$, as defined in Lemma~\ref{lemma:key},
converge to some solution \aor{$x^*\in Z^*$} and then show that $\bz_i(t+1)-\bar\bx(t)$ converges to 0 
for all $i$, as $t\to\infty$. The last step will be accomplished by invoking Lemma \ref{lemma:newx} on the
perturbed push-sum protocol. 

\begin{proof}[Proof of Theorem \ref{mainthm}] 
We begin by observing that the subgradient-push method
may be viewed as an instance of the perturbed push-sum protocol. 
Indeed, let us adopt the notation $\widetilde x_\ell(t), \widetilde g_\ell(t)$ from 
the proof of Lemma \ref{lemma:key}, and moreover let us define 
$\widetilde w_\ell(t), \widetilde z_\ell(t)$ \aor{similarly}. Then, 
the definition of subgradient-push method implies that for all $\ell=1,\ldots,d$,
\begin{eqnarray*} 
&& \widetilde w_\ell (t+1)  =  A(t) \widetilde x_\ell(t) ,\cr
&& y(t+1)  = A(t) y(t), \cr
&& [{\widetilde z}_\ell(t+1)]_i  =  \frac{[\widetilde w_\ell(t+1)]_i}{y_i(t+1)}\qquad \aor{\hbox{ for } i = 1, \ldots, n}, \\
&& \widetilde x_\ell(t+1)  = \widetilde w_\ell(t+1) - \alpha(t+1) \widetilde g_\ell(t+1).
\end{eqnarray*} 

Next, we \aor{will} apply Lemma~\ref{lemma:newx}(b$)$ to each coordinate $\ell=1,\ldots,d$.
Since $\alpha(t) \rightarrow 0$ and the subgradients are bounded, the assumptions of 
Lemma~\ref{lemma:newx}(b) are satisfied with 
$\e(t+1)= \alpha(t+1) \widetilde g_\ell(t+1)$ for each $\ell$. 
From Lemma \ref{lemma:newx}(b) we conclude that
\begin{equation*}
\lim_{t\to\infty} \left |[\widetilde{z}_\ell(t+1)]_i -\frac{\sum_{j=1}^n [\widetilde x_{\ell}(t)]_j}{n}\right|=0
\end{equation*}
for all $\ell=1,\ldots,d$ and all $i=1,\ldots,n,$
which is equivalent to
\begin{equation}\label{eq:con0}
\lim_{t\to\infty} \|\bz_{i}(t+1) -\bar\bx(t)\|=0
\quad\hbox{for all $i=1,\ldots,n.$}\end{equation}

Now, we \aor{will} apply Lemma~\ref{lemma:newx}(c$)$ to each coordinate $\ell=1,\ldots,d$.
Let $\ell$ be arbitrary but fixed coordinate index.
Since the subgradients $\bg_i(s)$ are uniformly bounded,
using $\e(t+1)= \alpha(t+1) \widetilde g_\ell(t+1)$ 
we can see that for all $i=1,\ldots,n$ and all $t\ge1$,
\[|\e_{i}(t)|  \le  \a(t) \|\widetilde g_\ell(\aor{t})\|_{\infty} 
\leq \a(t) \max_i\| \bg_i(\aor{t})\|.\]
Therefore, since by assumption $\sum_{t=1}^\infty \a(t)^2<\infty$ and $\|\bg_i(t)\|\le L_i$, 
we obtain for all $i=1,\ldots,n$,
\[\sum_{t=1}^\infty \a(t)|\e_{i}(t)| 
\leq \left(\max_i L_i\right)\sum_{t=1}^{\infty} \a^2(t)  < \infty.\]
In view of the preceding relation and the assumption that the sequence $\{\alpha(t)\}$ is non-increasing, 
by applying Lemma~\ref{lemma:newx}(b) to each coordinate $\ell=1,\ldots,d$,
we obtain
\[\sum_{t=0}^\infty\a(t+1) \left| [\widetilde z_\ell(t+1)]_i  
- \frac{\sum_{j=1}^n [\widetilde x_{\ell}(t)]_j}{n}\right|<\infty,\]
for all $\ell=1,\ldots,d$ and all $i=1,\ldots,n$,
 implying that
\begin{equation}\label{eq:sumfin}
\sum_{t=0}^\infty\a(t+1) \|\bz_{i}(t+1) - \bar\bx(t)\|<\infty
 \quad \hbox{for all $i$.}
\end{equation}

Now, \aor{observe that we can apply} Lemma~\ref{lemma:key} \aor{with}
$\bv=z^*$ for \aor{any} solution $z^*\in Z^*$ \aor{to obtain}
\begin{align}\label{eq:beg6}
\|\bar\bx(t+1) & - z^* \|^2 
\le  \|\bar\bx(t) - z^* \|^2 \cr
& - \frac{2\alpha(t+1)}{n} \left(F(\bar\bx(t)) - F^* \right) \cr
&  + \frac{4\alpha(t+1)}{n} \sum_{i=1}^n L_i \|\bz_i(t+1)-\bar\bx(t)\|\cr
& +\alpha^2(t+1)\frac{\aor{L^2}}{n^2},\end{align}
where $F^*$ is the optimal value (i.e., $F^*= F(z^*)$ for any $z^*\in Z^*$).
In view of Eq.~\eqref{eq:sumfin}, it follows that 
\[\sum_{t=0}^\infty \frac{4\alpha(t+1)}{n} \sum_{i=1}^n L_i \|\bz_i(t+1)-\bar\bx(t)\|<\infty.\]
Also, by assumption we have $\sum_{t=1}^\infty \alpha(t)=\infty$ and 
$\sum_{t=1}^\infty \alpha^2(t)<\infty$.
Thus, the conditions of Lemma~\ref{lemma:opt} are satisfied 
\aor{(note that $F(\cdot)$ is continuous since it is convex over $\mathbb{R}^d$)}, and by this lemma 
we conclude that the average sequence $\{\bar\bx(t)\}$ converges to a solution $\widehat z \in Z^*$.
By Eq.~\eqref{eq:con0} it follows that each sequence $\{\bz_i(t)\}$, $i=1,\ldots,n,$ 
converges to the same solution~$\widehat z$.
\end{proof}

Having proven Theorem \ref{mainthm}, we now turn to the proof of the convergence rate results of 
Theorem~\ref{convrate}. 
The first step will be a slight modification of the result we have just proved - whereas the proof of 
Theorem~\ref{mainthm} relies on $F(\bar \bx(t)) \rightarrow F^*$, \aor{in order to obtain a convergence 
rate} we will now need to argue that we can replace $F(\bar \bx(t))$ by $F$ evaluated
at a running average of the vectors $\bz_i(t)$ for any $i$. This is stated precisely in the \aor{following} lemma. 

\begin{lemma} \label{lemma:runavgbound} 
If all the assumptions of Theorem \ref{mainthm} are satisfied and  $\alpha(t) = 1/\sqrt{t}$, then  
for all $t\ge1$ and any $z^*\in Z^*$,
\begin{eqnarray*}  
&& F \left( \frac{\sum_{k=0}^t \alpha(k+1) \bar\bx(k) }
{\sum_{k=0}^t \alpha(k+1)} \right) - F(z^*) \cr
&&\hbox{}\cr
&&\leq  \frac{n}{2}\frac{\|\aor{\bar\bx(0)} - \bz^*\|_1}{\sqrt{t+1} } 
+ \frac{\aor{L^2} \left( 1 + \ln (t+1) \right)}{2 n\sqrt{t+1} } \cr
&&\hbox{}\cr
 & & + \frac{ 16 \aor{L} \sum_{j=1}^n \|\bx_j(0)\|_1 }
 {   \delta (1-\lambda) \sqrt{t+1} }\cr 
 &&\hbox{}\cr
 &&+ \frac{16 \aor{d \aor{L^2}} \left( 1+ \ln t \right) }{\delta (1-\lambda) \sqrt{t+1} }. 
 \end{eqnarray*}  
\end{lemma}

\begin{proof} 
From Lemma~\ref{lemma:key} 
we have for any $\bv$ and $t\ge0$,
\begin{eqnarray*}
&&\sum_{k=0}^t \frac{2\alpha(k+1)}{n} \left(F(\bar\bx(k)) - F(\bv) \right)
\le  \|\bar \bx(0) - \bv\|^2  \cr
&&+ \sum_{k=0}^t \frac{4\alpha(k+1)}{n} \sum_{i=1}^n L_i \|\bz_i(k+1)-\bar \bx(k)\|\cr 
&& +\sum_{k=0}^t \alpha^2(k+1)\frac{\aor{L^2}}{n^2}.\end{eqnarray*}  
From the preceding relation, \ao{recalling the notation $S(t)=\sum_{k=0}^{t-1} \alpha(k+1)$} and dividing by $(2/n)S(t+1)$,
we obtain for any $\bv \in \R^d$ and $t\ge0$, 
 \begin{eqnarray*}
&&\frac{\sum_{k=0}^t\alpha(k+1) F(\bar\bx(k))}{S(t+1)} - F(\bv)
\le  \frac{n}{2}\frac{\|\bar \bx(0) - \bv\|^2 }{S(t+1)}\cr
 &&+ \frac{2}{S(t+1)}\sum_{k=0}^t\alpha(k+1)
\sum_{i=1}^n L_i \|\bz_i(k+1)-\bar \bx(k)\|\cr
 &&+\frac{1}{S(t+1)}\sum_{k=0}^t \alpha^2(k+1)\frac{\aor{L^2}}{2n}.\end{eqnarray*}  
 Now setting $\bv = \bz^*$ for some $\bz^* \in Z^*$ and using the convexity of $F$, we obtain 
 \begin{eqnarray} \label{eq:i1} 
&& F \left( \frac{\sum_{k=0}^t \alpha(k+1) \bar\bx(k) }{S(t+1)} \right) - F(\bz^*) \cr
  &&\leq  \frac{n}{2}\frac{\|\bar \bx(0) - \bz^*\|^2 }{S(t+1)} \cr
  && + \frac{2}{S(t+1)}\sum_{k=0}^t\alpha(k+1)
\sum_{i=1}^n L_i \|\bz_i(k+1)-\bar \bx(k)\|\cr
 &&+\frac{1}{S(t+1)}\sum_{k=0}^t \alpha^2(k+1)\frac{\aor{L^2}}{2n}.
 \end{eqnarray}  
 We now bound the \aor{quantities} $\|\bz_i(\aor{k}+1)-\bar \bx(\aor{k})\|$ in Eq.~\eqref{eq:i1} 
 by applying Corollary~\ref{cor:sqrtavg} to each of its components. Specifically, we will
 apply Eq.~(\ref{eq:zxdiff}) of the corollary. Observe that all the assumption of that corollary have been assumed to hold. Moreover, the constant $D$ in the statement of the corollary \aor{can be taken to be $\sum_{i=1}^n L_i$} when 
 the corollary is applied to the $\ell$'th component of the vector
 $\bz_i(t+1)-\bar \bx(t)$. Adopting our notation of 
 $\widetilde{x}_\ell(t)$ and $\widetilde{z}_\ell(t)$ as the vectors consisting of 
 the $\ell$'th components of the vectors $\bx_i(t)$ and $\bz_i(t)$, $i=1,\ldots,n$,
 respectively, we have for all $\ell=1, \ldots, d$, all $i=1,\ldots,n$, and $t\ge1$,
 \begin{align*}
 &\sum_{k=0}^t \frac{1}{\sqrt{k+1}} 
 \left| [\widetilde z_\ell(k+1)]_i - \frac{1}{n} \sum_{j=1}^n [\widetilde x_\ell(k)]_j\right| \cr
 & \leq \frac{8}{\delta(1-\lambda)} \left(\|\widetilde x_\ell(0)\|_1 + \aor{L} (1 + \ln t) \right)\end{align*}
  (cf.\ Eq.~(\ref{eq:zxdiff}) of Corollary~\ref{cor:sqrtavg}),
 which after some elementary algebra implies that 
 \begin{align} \label{eq:zxdiffbd} 
 &\sum_{k=0}^t \alpha(k+1) 
 \sum_{i=1}^n L_i \|\bz_i(k+1)-\bar \bx(k)\|_1 \cr
 & \leq  \frac{8} {\delta(1-\lambda)} \aor{L} \sum_{j=1}^n \|\bx_j(0)\|_1 \cr
 & +\frac{8 \aor{d}} {\delta(1-\lambda)} \aor{L^2} \left( 1+ \ln t \right).
 \end{align} 
 Now, using the fact that the Euclidean norm of a vector is not larger that its $1$-norm, we substitute 
 Eq.~(\ref{eq:zxdiffbd}) into Eq.~(\ref{eq:i1}). Then, using the definition of $S(t)$ and Eq.~(\ref{eq:sqrtsum})
 we bound the denominator in Eq.~(\ref{eq:i1}) as follows
 \[ S(t+1)=\sum_{k=0}^t \alpha(k+1) \geq \sqrt{t+1} .\] The preceding relation and 
\[\sum_{k=0}^t\alpha^2(k+1)=\sum_{s=1}^{t+1}\frac{1}{s}\le 1 + \int_1^{t+1}\frac{dx}{x}= 1+\ln (t+1),\]
yield the stated estimate.
\end{proof}

We are now finally in position to prove Theorem~\ref{convrate}. 
At this point, the proof is a simple combination of Lemma~\ref{lemma:runavgbound}, which tells us that 
$F \left( \frac{\sum_{k=0}^t \alpha(k+1) \bar \bx(k)}{\sum_{k=0}^t \alpha(k+1)} \right)$ 
approaches $F(z^*)$, along with Corollary~\ref{cor:sqrtavg}, which tells us that 
 $\frac{\sum_{k=0}^t \alpha(k+1) \bar \bx(k)}{\sum_{k=0}^t \alpha(k+1)}$ and 
 $\frac{\sum_{k=0}^t \alpha(k+1)  \bz_i(k+1)}{\sum_{k=0}^t \alpha(k+1)}$ get close to 
 each other over time for all $i=1, \ldots, n$. 

\begin{proof}[Proof of Theorem \ref{convrate}] It is easy to see by induction that 
the vectors $\widetilde \bz_i(t)$ defined in the statement of 
Theorem~\ref{convrate} are weighted time-averages of $\bz_i(t)$, given by: for all $i$,
\[ \widetilde \bz_i(t+1) = \frac{\sum_{k=0}^t \alpha(k+1) \bz_i(k+1)}{\sum_{k=0}^t \alpha(k+1)}
\quad\hbox{for all $t\ge0$}. \] 
By the boundedness of subgradients we obtain for all $i$ and $t\ge0$,
\begin{align*} 
&F( \widetilde \bz_i(t+1) ) - F \left( \frac{\sum_{k=0}^t \alpha(k+1) \bar\bx(k) }
{\sum_{k=0}^t \alpha(k+1)} \right)\cr
&\leq \frac{\aor{L} }
{\sum_{k=0}^t \alpha(k+1)}\sum_{k=0}^t\a(k+1)\|\bz_i(k+1)-\bar \bx(k)\|.
\end{align*}
Applying Corollary~\ref{cor:sqrtavg} to each of the coordinates of the vectors $\bz_i(k+1)$ and 
$\bar \bx(t)$, we obtain for all $i$ and $t\ge1$,
\begin{align*}  
& F( \widetilde \bz_i(t+1) ) - F \left( \frac{\sum_{k=0}^t \alpha(k+1) \bar\bx(k) }
{\sum_{k=0}^t \alpha(k+1)} \right) \cr
&\leq \frac{8 \aor{L}} {\delta(1-\lambda) \sqrt{t+1} } 
\left( \sum_{j=1}^n \|\bx_j(0)\|_1 
+ \aor{d} \aor{L} (1 + \ln t) \right). 
\end{align*}
By summing the preceding relation and that of Lemma \ref{lemma:runavgbound},  we have
\begin{eqnarray*}  
&& F \left( \widetilde \bz_i(t+1) \right) - F(z^*) 
\leq  \frac{n}{2}\frac{\|\aor{\bar\bx(0)} - \bz^*\|_1}{\sqrt{t+1}} \cr
&&+ \frac{\aor{L^2} \left( 1 + \ln (t+1) \right)}{2 n\sqrt{t+1} } \cr
&&\hbox{}\cr
 & & + \frac{ 24 \aor{L} \sum_{j=1}^n \|\bx_j(0)\|_1 }
 {   \delta (1-\lambda) \sqrt{t+1} }\cr 
 &&\hbox{}\cr
 &&+ \frac{24 \aor{d} \aor{L^2} \left( 1+ \ln t \right) }{\delta (1-\lambda) \sqrt{t+1} }, 
 \end{eqnarray*}
thus proving  Theorem~\ref{convrate}.
\end{proof} 

\section{Simulations\label{sec:simul}} 
We report some simulations of the subgradient-push method which 
experimentally demonstrate that its performance is often quite scalable.  
We will be optimizing the scalar function $F(\theta) = \sum_{i=1}^n p_i (\theta - u_i)^2$ 
where $u_i$ is a variable that is known only to node $i$. 
This is a canonical problem in distributed estimation: 
the nodes are attempting to measure a parameter $\widehat \theta$, and 
node $i$ measures $u_i = \widehat \theta + w_i $ where $w_i$ are jointly Gaussian and zero mean. 
Letting $p_i$ be the inverse of the variance
of $w_i$, the maximum likelihood estimate is the minimizer $\theta^*$ of $F(\theta)$ 
($\theta^*$ is unique provided that $p_i>0$ for at least one $i$). 
We set a half of the $p_i$'s to zero 
(corresponding to nodes not taking any measurements) and 
the other half of the $p_i$'s to uniformly random values between $0$ and $1$. 
The initial points $x_i(0)$ are independent random variables, each with a standard 
Gaussian distribution. 
Figure~1 shows the results for simple random graphs where every node has two out-neighbors, 
one belonging to a fixed cycle and the other one chosen uniformly at random at each step. 
The plot to the left shows how \aor{$\|z(t) - \theta^*\1\|$} decays on a graph instance with $n=1000$ nodes, while one to the right shows the (average) time until \aor{$\|z(t) - \theta^*\1\| \leq 0.1$} 
as the size $n$ of the graph varies from 10 to 90. 
Figure 2 illustrates the same quantities for the sequence of graphs which alternate between 
two (undirected) star graphs. 
\begin{figure}[t!]
\centering
\includegraphics[scale=0.4]{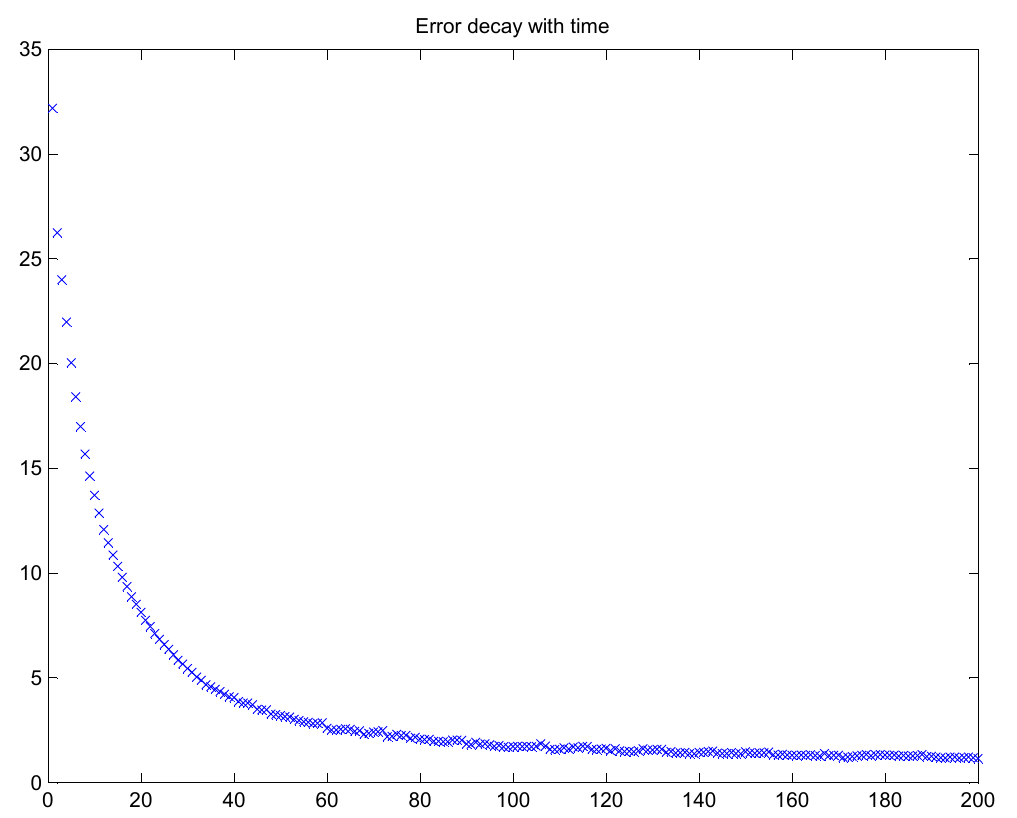}
\includegraphics[scale=0.4]{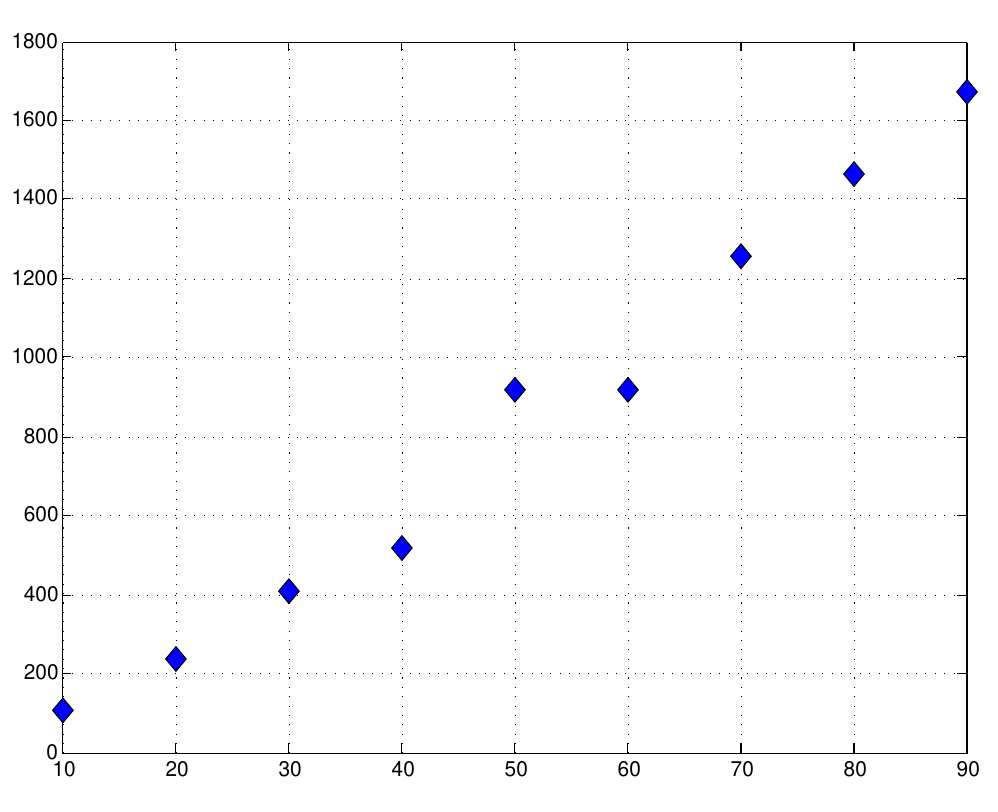}
\caption{Top plot: the number of iterations ($x$-axis) and $\|z(t) - \aor{\theta^*}\1\|$ ($y$-axis), where $z(t)$ is the vector stacking up all $z_i(t)$,
for an instance of a graph with 1000 nodes. 
Bottom plot: the number of nodes ($x$-axis) and 
the average time until  $\|z(t) - \aor{\theta^*}\1\| \leq 0.1$ over 30 simulations ($y$-axis).}
\label{f1}
\end{figure}

\begin{figure}[h!]
\centering
{\includegraphics[scale=0.4]{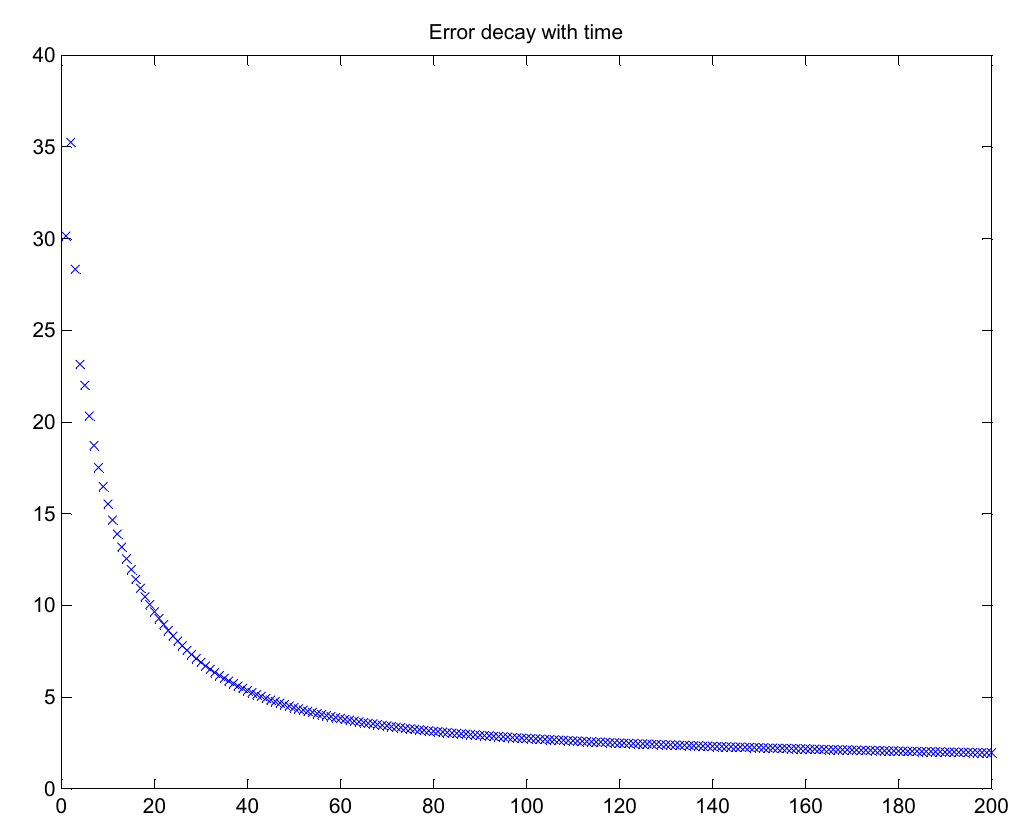}}
{\includegraphics[scale=0.4]{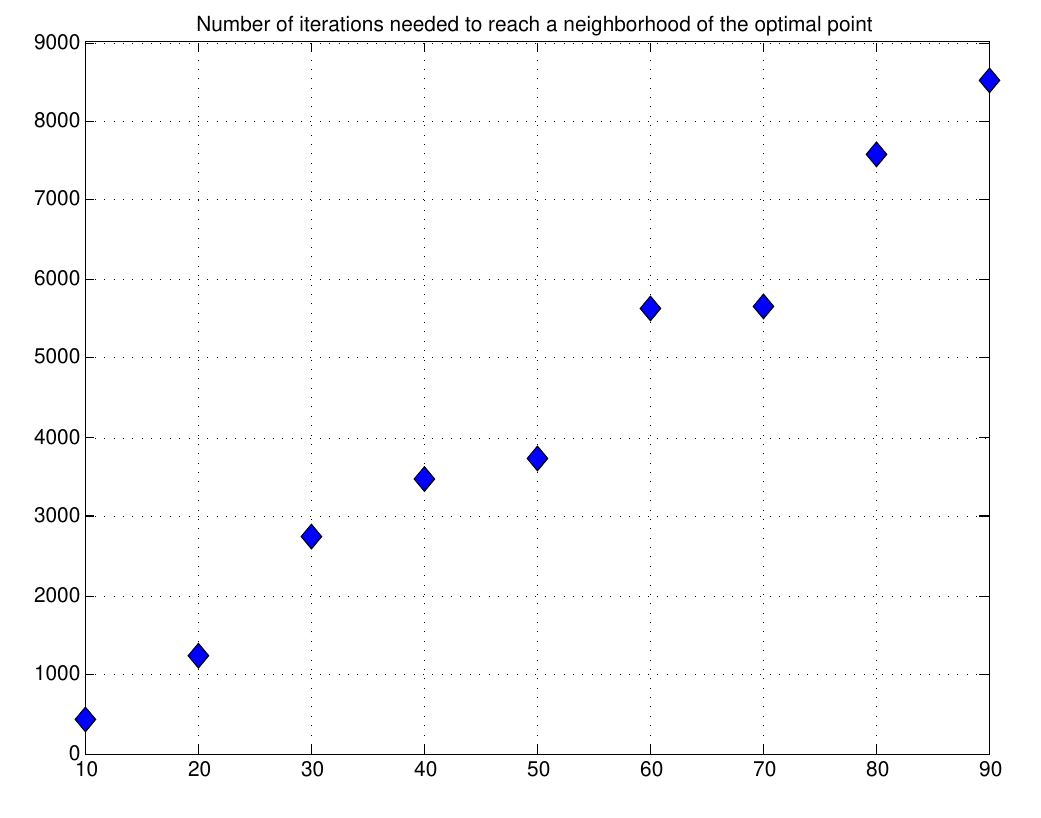}}
\caption{Top plot: the number of iterations ($x$-axis) and $\|z(t) - \aor{\theta^*}\1\|$ ($y$-axis), where $z(t)$ is the vector stacking up all $z_i(t)$,
for an instance of a graph with $n=1000$ nodes. 
Bottom plot: the number of nodes ($x$-axis) and the average time until 
$\|z(t) - \theta^*\1\|\leq 0.1$ over 50 simulations ($y$-axis).}
\end{figure}

We see that in both cases the initial decay
is quite fast and it takes only a small number of iterations to bring all the nodes reasonably close to the optimal value. Nevertheless, the decay is
clearly sub-geometric, consistently with the bounds we have proven.  However,
in both cases, the time to get the error below the threshold of $0.1$ starting from random initial points 
appears to scale linearly in the number of nodes. 

\section{Conclusions\label{sec:concl}} 
We have introduced the subgradient-push, a broadcast-based distributed protocol for distributed optimization of a sum of convex 
functions over directed graphs. We have shown that, as long as the communication graph
sequence $\{G(t)\}$ is uniformly strongly connected, the subgradient-push succeeds in driving all the nodes to
the same point in the set of optimal solutions. Moreover, the objective function converges at a rate of 
$O( \ln t/\sqrt{t})$, where the constant depends on the initial vectors,
bounds on subgradient norms, consensus speed $\lambda$ of the graph sequence $\{G(t)\}$, as well as a measure of the imbalance of influence $\delta$ among the nodes. 

Our results motivate the open problems associated with understanding how the consensus 
speed $\lambda$ depends on properties of the sequence $\{G(t)\}$. 
Similarly, it is also natural to ask how the measure of imbalance of influence $\delta$ depends on the combinatorial properties of the graphs $G(t)$, 
namely how it depends on the diameters, the size of the smallest cuts, and other pertinent features of the graphs in the sequence 
$\{G(t)\}$.

\bibliographystyle{plain}
\bibliography{directed-opt}

\end{document}